\documentclass[10pt]{amsart}

\usepackage[foot]{amsaddr}
\usepackage[latin1]{inputenc}

\usepackage{amsmath,amssymb,amsthm,amscd}
\usepackage{colordvi,graphicx}
\usepackage[margin=1in]{geometry}
\usepackage{hyperref}
\usepackage[enableskew]{youngtab}
\usepackage[all,cmtip]{xy}
\usepackage{color}
\usepackage{qtree}
\usepackage{amsfonts}
\usepackage{amsaddr}
\usepackage{enumitem}
\usepackage{tikz}
\usetikzlibrary{math,calc}
\usepackage{hyphenat}
\usepackage{booktabs}

\newtheorem{theorem}{Theorem}[section]
\newtheorem{lemma}[theorem]{Lemma}
\newtheorem{proposition}[theorem]{Proposition}

\theoremstyle{definition}
\newtheorem{definition}[theorem]{Definition}

\newtheorem{remark}[theorem]{Remark}

\numberwithin{equation}{section}
\numberwithin{table}{section}
\numberwithin{figure}{section}

\def\CC{{\mathbb C}}

\newcommand{\IN}{\ensuremath{\mathbb{N}}}

\newcommand{\op}[1]{{#1}^\mathrm{op}} 

\newcommand{\SC}{\ensuremath{\mathsf{SC}}} 
\newcommand{\ID}{\ensuremath{\mathsf{ID}}} 

\newcommand{\B}{\ensuremath{\mathcal{B}}} 
\newcommand{\F}{\ensuremath{\mathcal{F}}} 

\newcommand{\spac}[2][n]{\ensuremath{s^\mathrm{ac}_{#1}({#2})}}

\newcommand{\spa}[2][n]{\ensuremath{s^\mathrm{a}_{#1}({#2})}}

\renewcommand{\SS}{{\mathfrak S}}

\def\RR{\mathbb{R}}

\DeclareMathOperator{\var}{var}

\usepackage{soul}

\usepackage{fancybox,framed}

\title{Associative-commutative spectra for some varieties of groupoids}

\author{Jia Huang}
\address[J. Huang]{Department of Mathematics and Statistics, University of Nebraska, Kearney, NE 68849, USA}
\email{huangj2@unk.edu}

\author{Erkko Lehtonen}

\address[E. Lehtonen]%
     {Department of Mathematics \\
     Khalifa University \\
     P.O.~Box 127788 \\
     Abu Dhabi \\
     United Arab Emirates
    }
\email{erkko.lehtonen@ku.ac.ae}

\thanks{}

\keywords{Associative-commutative spectrum; associative spectrum; binary operation; tree, $3$-element groupoid}

\begin{document}
\begin{abstract}
The associative spectrum of a groupoid (i.e., a set with a binary operation) measures its nonassociativity while the associative-commutative spectrum measures both nonassociativity and noncommutativity of the groupoid.
The two spectra are also the coefficients of the Hilbert series of certain operads. 
We establish upper bounds for the two spectra of various varieties of groupoids defined by different sets of identities and provide examples (often groupoids with three elements) for which the upper bounds are achieved.
Our results have connections to many interesting combinatorial objects and integer sequences and naturally lead to some questions for future studies.
\end{abstract}

\maketitle

\section{Introduction}\label{sec:Intro}

A \emph{groupoid} $(G,*)$ is a basic algebraic structure that consists of a set $G$ together with a binary operation $*$ defined on $G$.
Associativity and commutativity are common properties that could be satisfied by a groupoid.
Cs\'{a}k\'{a}ny and Waldhauser~\cite{AssociativeSpectra1} defined the \emph{associative spectrum} (also called the \emph{subassociativity type} by Braitt and Silberger~\cite{Subassociative}) to measure the failure of a groupoid to be associative, and we introduced the \emph{associative-commutative spectrum}, or simply \emph{ac-spectrum}, to measure both nonassociativity and noncommutativity of a groupoid in earlier work~\cite{AC-Spectrum}; see the definition below.

\begin{definition}
Fix a countable list of distinct variables $x_1, x_2, \ldots$.
Let $\B_n$ denote the set of all \emph{bracketings} of $x_1, \ldots, x_n$, which are terms in the language of groupoids obtained by inserting pairs of parentheses into the word $x_1 x_2 \cdots x_n$ in all valid ways.
Let $\F_n$ denote the set of \emph{full linear terms} over $x_1, \ldots, x_n$, which are obtained by permuting the variables in the bracketings of $x_1, \ldots, x_n$.
We can view $\B_n$ as a subset of $\F_n$.
Every term $t\in \F_n$ induces an $n$-ary operation $t^*$ on a groupoid $(G,*)$.
It is often convenient to think about the terms in $\F_n$ or the $n$-ary operations induced by them in terms of the corresponding \emph{\textup{(}ordered, full\textup{)} binary trees} with $n$ labeled leaves;
see the example below for $\B_4$, which can give $\F_4$ if the variables are permuted in all possible ways.
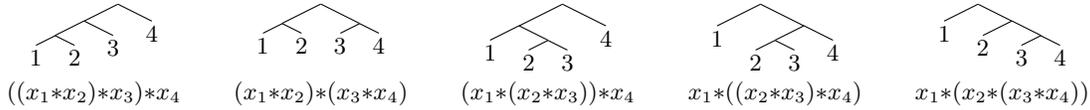
\[\small
\begin{array}{ccccc}
\Tree [.  [. [. 1 2 ] 3 ] 4 ] &
\Tree [.  [. 1 2 ] [. 3 4 ] ] &
\Tree [.  [. 1 [. 2 3 ] ] 4 ] &
\Tree [. 1 [. [. 2 3 ] 4 ] ] &
\Tree [. 1  [. 2 [. 3 4 ] ] ] \\
\rule{5pt}{0pt}((x_1 {*} x_2) {*} x_3) {*} x_4\rule{5pt}{0pt} &
\rule{5pt}{0pt}(x_1 {*} x_2) {*} (x_3 {*} x_4)\rule{5pt}{0pt} &
\rule{5pt}{0pt}(x_1 {*} (x_2 {*} x_3)) {*} x_4 \rule{5pt}{0pt} &
\rule{5pt}{0pt}x_1 {*} ((x_2 {*} x_3) {*} x_4)\rule{5pt}{0pt} &
\rule{5pt}{0pt}x_1 {*} (x_2 {*} (x_3 {*} x_4))\rule{5pt}{0pt}
\end{array}\]
The \emph{associative spectrum} (resp., \emph{ac-spectrum}) of a groupoid $(G,*)$, or of its binary operation $*$, is a sequence whose $n$th term is $\spa{*}:=|P_n(*)|$ (resp., $\spac{*}:=|\overline{P}_n(*)|$), where $P_n(*):=\{t^*: t\in \B_n\}$ (resp.,  $\overline P_n(*):=\{t^*: t\in \F_n\}$), for $n=1,2,\ldots$.
It turns out that $\{P_n(*)\}_{n\ge1}$ (resp., $\{\overline P_n(*)\}_{n\ge1}$) together with a composition function becomes a \emph{nonsymmetric operad} (resp., \emph{symmetric operad}) that satisfies certain coherence axioms~\cite{Operads}, and the \emph{Hilbert series} of this operad is the generating function (resp., exponential generating function) of the associative spectrum (resp., ac-spectrum) of $(G,*)$.
\end{definition}

By the above definition, we have (1) $\spa{*}=1$ for $n=1,2$, (2) $\mathrm{s^{ac}_1}(*)=1$, and (3) $\mathrm{s^{ac}_2}(*)$ is either $1$ or $2$, depending on whether $*$ is commutative.
Thus we may assume $n\ge3$ when necessary.
It is easy to see that isomorphic or anti-isomorphic groupoids have the same associative spectrum and the same ac-spectrum, where two groupoids $(G,*)$ and $(H, \otimes)$ are said to be \emph{anti-isomorphic}, denoted by $G \simeq \op{H}$, if there is a bijection $f: G\to H$ such that $f(a*b) = f(b) \otimes f(a)$ for all $a,b\in G$.

It is clear that $\spa{*}=1$ for all $n \in \IN$ if and only if $*$ associative and that $\spac{*}=1$
for all $n \in \IN$ if and only if $*$ is associative and commutative, where $\IN := \{1,2,\ldots\}$.
On the other hand, we have $\spa{*}\le C_{n-1}$, where $C_n:=\frac1{n+1} \binom{2n}{n}$ is the ubiquitous \emph{Catalan number}, and thus $\spac{*}\le n!C_{n-1}$.
We showed in previous work~\cite{AC-Spectrum} that a commutative groupoid $(G,*)$ must have $\spac{*}\le D_{n-1}$, where $D_n:=(2n!)/(2^n n!)$ is the solution to Schr\"oder's third problem~\cite[A001147]{OEIS}, and that an associative groupoid $(G,*)$ must have $\spac{*}\le n!$, which holds as an equality if the groupoid is noncommutative and has an identity element
(see Theorem~\ref{thm:n!} for a generalization).

In addition, the precise values of the associative spectrum and ac-spectrum have been determined for various groupoids~\cite{AssociativeSpectra1, CatMod, VarCat, AC-Spectrum, GraphAlgebra1, GraphAlgebra2}, including $2$-element groupoids, generalizations of addition and subtraction, exponentiation, arithmetic/geometric/harmonic mean, cross product, Lie algebras with an $\mathfrak{sl}_2$-triple, graph algebras, and so on.
The results show connections with interesting combinatorial objects, avoided patterns, and integer sequences.
However, the ac-spectra of $3$-element groupoids are largely undetermined.

According to the Siena Catalog~\cite{Siena}, there are $3330$ non-isomorphic $3$-element groupoids, which are indexed from $1$ to $3330$.
Each of these groupoids is determined by a binary operation $*$ defined on the set $\{0,1,2\}$.
We write them as $\SC1, \SC2, \ldots, \SC3330$.
There are $729$ \emph{idempotent} $3$-element groupoids, which can be labeled in a different way: $\ID0, \ID1, \ldots, \ID728$.
Cs\'{a}k\'{a}ny and Waldhauser~\cite{AssociativeSpectra1} showed the following (see Table~\ref{tab:3-element groupoids}).
\begin{itemize}
\item
Both $\ID35 = \SC271 (\simeq \op{\SC1610})$ and $\ID68 = \SC356 (\simeq \op{\SC2032})$ have associative spectrum $\spa{*} = 2^{n-2}$ for $n\ge2$.
\item
Both $\SC1066$ and $\SC10 (\simeq \op{\SC367})$ have associative spectrum $\spa{*} = n-1$ for $n\ge1$.
\item
Both $\SC405$ and $\SC3242 (\simeq \op{\SC3302})$ have associative spectrum $\spa{*} = 3$ for $n>3$ (it is easy to check that $\spa{*} = 1$ for $n=1,2$ and $\spa{*} = 2$ for $n=3$).
\item
The groupoid $\SC79$ has associative spectrum $\spa{*} = F_{n+1}-1$ for $n\ge2$, where $F_{n+1}$ is the \emph{Fibonacci number} defined by $F_{n+1} := F_n + F_{n-1}$ for $n\ge1$ and $F_i=i$ for $i=0,1$,
\end{itemize}

Our original motivation for this work was to determine the ac-spectra of the above $3$-element groupoids, whose Cayley tables are given in Table~\ref{tab:3-element groupoids}.
However, we are able to establish more general results on various varieties of groupoids, where a \emph{variety} of groupoids axiomatized by a set $\Sigma$ of identities is the family of all groupoids satisfying the identities in $\Sigma$.
For each variety of groupoids considered in this paper, we establish an upper bound for the associative spectra and an upper bound for the ac-spectra of the groupoids belonging to this variety; if the latter upper bound is reached by a member of the variety, so is the former. 
Moreover, we show that both upper bounds are attained by at least one $3$-element groupoid.

\begin{table}[ht]
\begin{center}
\small
\begin{tabular}{ccccccc}
\begin{tabular}{c|ccc}
$*$ & 0 & 1 & 2 \\
\hline
0 & 0 & 0 & 0 \\
1 & 1 & 1 & 0 \\
2 & 2 & 2 & 2 
\end{tabular}
&
\begin{tabular}{c|ccc}
$*$ & 0 & 1 & 2 \\
\hline
0 & 0 & 0 & 0 \\
1 & 2 & 1 & 1 \\
2 & 1 & 2 & 2 
\end{tabular}
&
\begin{tabular}{c|ccc}
$*$ & 0 & 1 & 2 \\
\hline
0 & 0 & 0 & 2 \\
1 & 0 & 0 & 2 \\
2 & 2 & 2 & 1 
\end{tabular}
&
\begin{tabular}{c|ccc}
$*$ & 0 & 1 & 2 \\
\hline
0 & 0 & 0 & 0 \\
1 & 0 & 0 & 0 \\
2 & 1 & 0 & 0 
\end{tabular}
&
\begin{tabular}{c|ccc}
$*$ & 0 & 1 & 2 \\
\hline
0 & 0 & 0 & 1 \\
1 & 0 & 0 & 1 \\
2 & 1 & 1 & 0 
\end{tabular}
&
\begin{tabular}{c|ccc}
$*$ & 0 & 1 & 2 \\
\hline
0 & 1 & 1 & 1 \\
1 & 2 & 2 & 2 \\
2 & 0 & 0 & 0 
\end{tabular}
&
\begin{tabular}{c|ccc}
$*$ & 0 & 1 & 2 \\
\hline
0 & 0 & 0 & 0 \\
1 & 0 & 1 & 0 \\
2 & 0 & 0 & 1 
\end{tabular}
\\
$\SC271$ & $\SC356$ & $\SC1066$ & $\SC10$ & $\SC405$ & $\SC3242$ & $\SC79$ \\
$\simeq \op{\SC1610}$ & $\simeq \op{\SC2032}$ & $= \op{\SC1066}$ & $\simeq \op{\SC367}$ & $\simeq \op{\SC405}$ & $\simeq \op{\SC3302}$ & $\simeq \op{\SC79}$
\\[2ex]
\begin{tabular}{c|ccc}
$*$ & 0 & 1 & 2 \\
\hline
0 & 0 & 1 & 1 \\
1 & 0 & 1 & 2 \\
2 & 0 & 1 & 2 
\end{tabular}
&
\begin{tabular}{c|ccc}
$*$ & 0 & 1 & 2 \\
\hline
0 & 0 & 1 & 2 \\
1 & 0 & 1 & 2 \\
2 & 1 & 0 & 2 
\end{tabular}
&
&
\begin{tabular}{c|ccc}
$*$ & 0 & 1 & 2 \\
\hline
0 & 0 & 0 & 1 \\
1 & 0 & 0 & 0 \\
2 & 0 & 0 & 0 
\end{tabular}
&
&
\begin{tabular}{c|ccc}
$*$ & 0 & 1 & 2 \\
\hline
0 & 1 & 2 & 0 \\
1 & 1 & 2 & 0 \\
2 & 1 & 2 & 0 
\end{tabular}
&
\\
$\SC1610$ & $\SC2032$ & & $\SC367$ & & $\SC3302$ & 
\end{tabular}
\end{center}
\bigskip
\caption{Some $3$-element groupoids}\label{tab:3-element groupoids}
\end{table}

For example, we showed in earlier work~\cite{AC-Spectrum} that a commutative groupoid must have $\spac{*}\le D_{n-1}$ and if the equality in this upper bound holds, so does the equality in the upper bound $\spa{*}\le C_{n-1}$.
In the same paper, we showed that $\spac{*}=D_{n-1}$ for a $3$-element groupoid called the \emph{rock-paper-scissors groupoid}, which turns out to be isomorphic to $\SC1108$, and the proof is also valid for $\SC2407$ and $\SC3093$.
\begin{center}
\begin{tabular}{ccccccc}
\begin{tabular}{c|ccc}
$*$ & $0$ & $1$ & $2$ \\
\hline
$0$ & $0$ & $0$ & $2$ \\
$1$ & $0$ & $1$ & $1$ \\
$2$ & $2$ & $1$ & $2$ 
\end{tabular}
&
\begin{tabular}{c|ccc}
$*$ & $0$ & $1$ & $2$ \\
\hline
$0$ & $1$ & $0$ & $0$ \\
$1$ & $0$ & $2$ & $0$ \\
$2$ & $0$ & $0$ & $0$ 
\end{tabular}
&
\begin{tabular}{c|ccc}
$*$ & $0$ & $1$ & $2$ \\
\hline
$0$ & $1$ & $1$ & $0$ \\
$1$ & $1$ & $2$ & $0$ \\
$2$ & $0$ & $0$ & $1$ 
\end{tabular}
\\
$\SC1108$ & $\SC2407$ & $\SC3093$
\end{tabular}
\end{center}
Therefore, we have the following result.

\begin{theorem}[\cite{AC-Spectrum}]
A groupoid $(G,*)$ satisfying the identity $xy\approx yx$ must have $\spac{*}\le C_{n-1}$ and $\spac{*} \le D_{n-1}$ for $n=1,2,\ldots$, where the first inequality holds as an equality whenever the second does and both equalities hold for the $3$-element groupoids $\SC1108$, $\SC2407$, and $\SC3093$.
\end{theorem}

In this paper, we provide a series of results that are similar to the above one.
A summary of our results is given by Table~\ref{tab:summary}, where we use the well-known \emph{Bell number} $B_n$ counting partitions of the set $\{1,2,\ldots,n\}$ into unordered nonempty blocks, the \emph{restricted Bell number} $B_{n,m}$ counting partitions of $\{1,2,\ldots,n\}$ into unordered nonempty blocks of size at most $m$~\cite{Mezo}, and the \emph{ordered Bell number} or \emph{Fubini number} $B'_n$ counting partitions of $\{1,2,\ldots,n\}$ into ordered nonempty blocks~\cite[A000670]{OEIS}.
The ``$n\ge$'' column in Table~\ref{tab:summary} gives the smallest values of $n$ for which the upper bounds of $\spa{*}$ and $\spac{*}$ are valid and sharp.
Note that different varieties of groupoids in the table may have the same associative spectrum upper bound but different ac-spectrum upper bounds (the upper bounds for $\spac{*}$ in Prop.~\ref{prop:7} and Prop.~\ref{prop:405} are different when $n=3$).
Therefore, the ac-spectrum often offers a finer distinction than the associative spectrum between groupoids satisfying different sets of identities. 

\begin{table}[ht]
\renewcommand{\arraystretch}{1.2}
\begin{center}
\begin{tabular}{lccccc} 
\toprule
Identities satisfied by $(G,*)$ & $n\ge$ & $\spa{*}\le$ & $\spac{*}\le$ & Examples for $=$ & reference \\
\midrule 
(1) & $1,1$ & $1$ & $n$ & $\SC275 (\simeq \op{\SC2029})$ & Prop.~\ref{prop:275} \\
\hline
(3), (4), (5), (7) & $3,3$ & $2$ & $n+1$ & $\substack{\SC7 (\simeq \op{\SC4}) \\ \SC28 (\simeq \op{\SC5})}$ & Prop.~\ref{prop:7} \\
\hline
(2), (7), (15) & $4,4$ & $3$ & $n+1$ & $\SC405$ & Prop.~\ref{prop:405} \\
\hline
(3), (5), (7), (8), (9) & $3,3$ & $2$ & $2n$ & $\SC189 (\simeq \op{\SC170})$ & Prop.~\ref{prop:189} \\
\hline
(5), (7), (10), (11), (12), (16) & $4,4$ & $3$ & $3n$ & $\SC3242 (\simeq \op{\SC3302})$ & Prop.~\ref{prop:3242} \\
\hline
(5), (7), (11), (13), (17), (18) & $4,4$ & $4$ &$2n^2$ & $\SC3162 (\simeq \op{\SC2467})$ & Thm.~\ref{thm:3162} \\
\hline
(2), (7) & $2,2$ & $n-1$ & $2^{n-1}-1$ & $\SC1066$ & Prop.~\ref{prop:1066} \\
\hline
(4), (5), (7) & $2,1$ & $n-1$ & $n! + \displaystyle\sum_{k=0}^{n-3} k! \binom{n}{k}$ & $\SC367 (\simeq \op{\SC10})$ & Prop.~\ref{prop:367}\\
\hline
(3), (6), (14) & $2,2$ & $2^{n-2}$ & $2^n-2$ & $\SC2302 (\simeq \op{\SC2155})$ & Prop.~\ref{prop:subtration} \\
\hline
(3), (7), (12) & $2,2$ & $2^{n-2}$ & $n(2^{n-1}-1)$ & $\substack{\SC271 (\simeq \op{\SC1610}) \\ \SC356 (\simeq \op{\SC2032})}$ & Thm.~\ref{thm:A058877} \\ 
\hline
(2), (11) & $2,2$ & $F_{n+1}-1$ & $B_{n,2}-1$ & $\SC79$, $\SC1701$ & Prop.~\ref{prop:79} \\
\hline
(3), (5) & $2,1$ & $2^{n-2}$ & $nB_{n-1}$ & $\substack{\SC41 (\simeq \op{\SC398}) \\ \SC96 (\simeq \op{\SC1069})}$ & Thm.~\ref{thm:set-partition} \\
\hline
(5), (7) & $2,1$ & $2^{n-2}$ & $n B'_{n-1}$ & $\substack{\SC262 (\simeq \op{\SC1441}) \\ \SC1812 (\simeq \op{\SC1793}) \\ \SC2446 (\simeq \op{\SC2430}) }$ & Thm.~\ref{thm:ordered-set-partition} \\
\bottomrule
\end{tabular}

\vskip10pt
(1) $xy \approx x$ \quad
(2) $xy \approx yx$ \quad
(3) $(xy)z\approx (xz)y$ \quad
(4) $x(yz) \approx y(xz)$ \quad
(5) $x(yz) \approx x(zy)$ \quad
(6) $x(yz) \approx z(yx)$ 

\vskip5pt

(7) $w(x(yz)) \approx w((xy)z)$ \quad
(8) $(wx)(yz) \approx (w(xy))z$ \quad
(9) $w(x(yz)) \approx ((wx)y)z$ 

\vskip5pt

(10) $((wx)y)z \approx ((wy)x)z$ \quad
(11) $((wx)y)z \approx ((wx)z)y$ \quad 
(12) $(wx)(yz) \approx (wy)(xz)$ 

\vskip5pt

(13) $(w(xy))z \approx (w(xz))y$ \quad
(14) $w(x(yz)) \approx (w(xy))z$ \quad
(15) $(v(wx))(yz) \approx (vw)(x(yz))$

\vskip5pt

(16) $((vw)x)y)z \approx v(w(x(yz)))$ \quad
(17) $v(w(x(yz))) \approx ((v(wx))y)z$ \quad
(18) $((vw)(x(yz)) \approx (((vw)x)y)z$
\end{center}

\bigskip
\caption{Summary of results}\label{tab:summary}
\end{table}

It is sometimes convenient to use not only identities but other conditions to describe a family of groupoids satisfying certain upper bounds for their spectra.
Recall that every term $t\in\F_n$ corresponds to a binary tree with $n$ leaves labeled by $1,\ldots,n$.
Each leaf $i$ has its \emph{depth} $d_i(t)$ (resp.\ \emph{left depth} $\delta_i(t)$ or \emph{right depth} $\rho_i(t)$) defined as the number of edges (resp., left/right edges) in the unique path to the root of $t$.
By abuse of notation, we also speak of these three kinds of depths for the variables in $t$.
Previous work \cite{CatMod, AC-Spectrum} used the congruence modulo $m$ relation on depths to study the associative spectra and ac-spectra of certain groupoids, and some of the results there can be rephrased to include Proposition~\ref{prop:subtration} as a special case.
We can also generalize Proposition~\ref{prop:189} and Proposition~\ref{prop:3242} in a similar way.

The paper is structured as follows.
We give some basic definitions and properties on the associative spectrum and ac-spectrum in Section~\ref{sec:prelim}.
We establish some polynomial upper bounds and exponential upper bounds in Section~\ref{sec:poly} and Section~\ref{sec:expo}, respectively. 
We provide more upper bounds related to set partitions in Section~\ref{sec:set-partition}.
We use congruence on leaf depths in binary trees to provide generalizations of some of our results in Section~\ref{sec:depth}.
Finally, we make some remarks and pose some questions for future research in Section~\ref{sec:question}.

\section{Preliminaries}\label{sec:prelim}

We first give some notation and terminology.
A \emph{term}%
\footnote{More specifically, we are speaking about terms in the language of groupoids, i.e., terms of type $(2)$.
Since our language contains only one operation symbol, which is binary, we may simply omit it from terms.
Variables and brackets are sufficient for writing down terms unambiguously in this language.}
$t$ over a set of variables $X$ (we often use $X_n:=\{x_1,\ldots,x_n\}$) is a bracketing of a word $x_{i_1} \cdots x_{i_k}$, where $x_{i_1}, \ldots, x_{i_k} \in X$; let $\var(t)$ denote the set of all variables in $t$.
If $i_1, \ldots, i_k$ are distinct, then $t$ is called a \emph{linear term} with $|t|:=k$.
Define the \emph{leftmost bracketing} $[t_1,\ldots,t_k]$ of terms $t_1, \ldots, t_k$ recursively by $[t_1] := t_1$ and
$[t_1, \dots, t_{n+1}] := ([t_1, \dots, t_n]t_{n+1}])$ for $n \geq 1$.
Similarly, define the \emph{rightmost bracketing} $\langle t_1,\ldots,t_k\rangle$ recursively by $\langle t_1 \rangle := t_1$ and $\langle t_1, \dots, t_{n+1} \rangle := ( t_1 \langle t_2, \dots, t_{n+1} \rangle )$ for $n \geq 1$.
We can write every term as $t = [t_0, t_1, \ldots, t_m]$ with $|t_0|=1$ for some $m \in \IN$;
this is known as the \emph{leftmost decomposition}~\cite[Definition~6.1.2]{AC-Spectrum}.

Terms can be evaluated in a groupoid $(G,*)$ as follows.
Given an \emph{assignment} $h \colon X \to G$ of values from $G$ for the variables in $X$,
we can extend $h$ to a map $\overline{h}$ defined on the set of all terms over $X$ with the following recursive definition.
We have $\overline{h}(x) := h(x)$ for every variable $x \in X$ (because $\overline{h}$ extends $h$), and
if $t = (t_1 t_2)$ for subterms $t_1$ and $t_2$, then we define $\overline{h}(t) := \overline{h}(t_1) * \overline{h}(t_2)$.
In this way, every term $t$ over $X_n$ \emph{induces} an $n$-ary operation $t^*$ on $(G,*)$ (called a \emph{term function}):
$t^*(a_1, \dots, a_n) := \overline{h}_\mathbf{a}(t)$, where $\overline{h}_\mathbf{a}$ is the extension of the assignment $h_\mathbf{a} \colon X_n \to G$ that maps $x_i$ to $a_i$ for all $i \in \{1, \dots, n\}$.
For notational simplicity, we will denote the extension $\overline{h}$ of an assignment $h$ also by $h$.

An \emph{identity} is a pair of terms, usually written as $s \approx t$.
A groupoid $(G,*)$ \emph{satisfies} an identity $s \approx t$ if $s^* = t^*$.
(Here we have assumed that $s$ and $t$ are terms over $X_n$ for some $n \in \IN$ -- this can always be done.)

In the subsequent sections, we will prove several results, each of which provides
upper bounds for the ac-spectrum and the associative spectrum of 
a \emph{variety} of groupoids axiomatized by a set $\Sigma$ of identities, i.e., the family of all groupoids satisfying the identities in $\Sigma$.
We will employ the following proof technique.
We assume that a groupoid $(G,*)$ satisfies certain identities.
Using these identities, we transform each full linear term $t$ into an equivalent term $t'$ that is in ``standard form'' (terms $t$ and $t'$ are \emph{equivalent} if $(G,*)$ satisfies $t \approx t'$, i.e., $t^* = (t')^*$).
It thus follows that $\spac{*}$, i.e., the number of term functions induced by full linear terms with $n$ variables on $(G,*)$, is bounded above by the number of terms in standard form, so it is then a matter of counting the possible standard forms.
Similarly, finding $\spa{*}$ amounts to counting the standard forms that can be obtained from bracketings.

Let $t$ be a linear term.
Assume that $\var(t) = \{x_{i_1}, \dots, x_{i_m}\}$ and that $x_{i_k}$ occurs to the left from $x_{i_\ell}$ in $t$ if and only if $k < \ell$.
Assume further that $\{j_1, \dots, j_m\} = \{i_1, \dots, i_m\}$ and $j_1 < j_2 < \dots < j_m$.
Let
\begin{align*}
t^\mathrm{L} &:= [ x_{i_1}, \dots, x_{i_m} ],  &
t^{\mathrm{L}\mathord{<}} &:= [ x_{j_1}, \dots, x_{j_m} ],  \\
t^\mathrm{R} &:= \langle x_{i_1}, \dots, x_{i_m} \rangle,  &
t^{\mathrm{R}\mathord{<}} &:= \langle x_{j_1}, \dots, x_{j_m} \rangle,
\end{align*}
i.e., $t^\mathrm{L}$ and $t^{\mathrm{L}\mathord{<}}$ ($t^\mathrm{R}$ and $t^{\mathrm{R}\mathord{<}}$, resp.) are leftmost (rightmost, resp.) bracketings of the variables of $t$; in the former, the variables occur in the same order as in $t$, while in the latter, the variables occur in the increasing order of the indices.

The next lemma will be frequently used to establish our main results.

\begin{lemma}\label{lem:t}
Let $(G,*)$ be a groupoid, and write an arbitrary term in $\F_n$ as $t = [t_0, t_1, \dots, t_m]$ with $|t_0| = 1$ (leftmost decomposition).
\begin{enumerate}[label={\upshape(\roman*)}]
\item\label{lem:t:rightmost}
If $(G,*)$ satisfies the identity $w(x(yz)) \approx w((xy)z)$, then $(G,*)$ also satisfies the identities $t \approx [t_0, t_1^\mathrm{L}, \dots, t_m^\mathrm{L}]$ and $t \approx [t_0, t_1^\mathrm{R}, \dots, t_m^\mathrm{R}]$.
\item\label{lem:t:rightmost-increasing}
If $(G,*)$ satisfies the identities $w(x(yz)) \approx w((xy)z)$ and either $x(yz) \approx x(zy)$ or $xy \approx yx$, then $(G,*)$ also satisfies the identities $t \approx [t_0, t_1^{\mathrm{L}\mathord{<}}, \dots, t_m^{\mathrm{L}\mathord{<}}]$ and $t \approx [t_0, t_1^{\mathrm{R}\mathord{<}}, \dots, t_m^{\mathrm{R}\mathord{<}}]$.
\item\label{lem:t:order-ti}
If $(G,*)$ satisfies the identity $(xy)z \approx (xz)y$, then $(G,*)$ also satisfies the identity $t \approx [t_0, t_{\sigma(1)}, \dots, t_{\sigma(m)}]$ for every permutation $\sigma \in \SS_m$.
\item\label{lem:t:rightmost-increasing:2}
If $(G,*)$ satisfies the identities $x(yz) \approx x(zy)$ and $(xy)z \approx (xz)y$, then $(G,*)$ also satisfies the identity $t \approx [t_0, t_1^{\mathrm{L}\mathord{<}}, \dots, t_m^{\mathrm{L}\mathord{<}}]$.
\end{enumerate}
\end{lemma}

\begin{proof}
\ref{lem:t:rightmost}
We can use the identity $w(x(yz)) \approx w((xy)z)$ repeatedly to transform each $t_i$ to the form $x_j s$, where $x_j$ is the leftmost variable of $t_i$, and then apply the same procedure to $s$ to eventually transform $t_i$ into $t_i^\mathrm{R}$.
A similar argument shows that each $t_i$ can be transformed into $t_i^\mathrm{L}$.

\ref{lem:t:rightmost-increasing}
By \ref{lem:t:rightmost}, $(G,*)$ satisfies $t \approx [t_0, t_1^\mathrm{R}, \dots, t_m^\mathrm{R}]$.
We may arbitrarily permute the variables in each $t_i^\mathrm{R}$, $i \in \{1, \dots, m\}$, thanks to the identities 
\[
   w(x(zy)) \approx w(x(yz))  \approx w((xy)z) \approx w(z(xy)) \approx w(z(yx)) \approx w((yx)z) \approx w(y(xz)).
\]
Thus $(G,*)$ satisfies $t \approx [t_0, t_1^{\mathrm{R}\mathord{<}}, \dots, t_m^{\mathrm{R}\mathord{<}}]$.
A similar argument shows that
$(G,*)$ satisfies $t \approx [t_0, t_1^{\mathrm{L}\mathord{<}}, \dots, t_m^{\mathrm{L}\mathord{<}}]$.

\ref{lem:t:order-ti}
We can use the identity $(xy)z \approx (xz)y$ to swap the subterms $t_i$ and $t_{i+1}$ in $[t_0, t_1, \dots, t_m]$, for any $i \in \{1, \dots, m-1\}$.
Since the adjacent transpositions generate $\SS_m$, it follows that $(G,*)$ satisfies $t \approx [t_0, t_{\sigma(1)}, \dots, t_{\sigma(m)}]$ for every $\sigma \in \SS_m$.

\ref{lem:t:rightmost-increasing:2}
By \ref{lem:t:order-ti}, we can permute the subterms $t_1, \dots, t_m$, so it suffices to prove that $(G,*)$ satisfies $x s \approx x (s^{\mathrm{L}\mathord{<}})$ for any linear term $s$ with $x \notin \var(s)$.
We prove this by induction on $|s|$.
This is trivial when $|s| = 1$, and this holds for $|s| = 2$ by the identity $x(yz) \approx x(zy)$.
Let now $k \geq 3$, assume that the claim holds whenever $|s| < k$, and consider the case when $|s| = k$.
We have the leftmost decomposition $s = [s_0, s_1, \dots, s_\ell]$.
By the inductive hypothesis and \ref{lem:t:order-ti}, we may assume that $s_j = s_j^{\mathrm{L}\mathord{<}}$ for all $j \in \{1, \dots, \ell\}$.
Consequently, $(G,*)$ satisfies $x s \approx x (s_\ell^{\mathrm{L}\mathord{<}} u)$, where $u := [s_0, s_1^{\mathrm{L}\mathord{<}}, \dots, s_{\ell - 1}^{\mathrm{L}\mathord{<}}]$, and by the inductive hypothesis, this is equivalent to $x (s_\ell^{\mathrm{L}\mathord{<}} u^{\mathrm{L}\mathord{<}})$.
By the identity $x(yz) \approx x(zy)$, we may swap $s_\ell^{\mathrm{L}\mathord{<}}$ and $u^{\mathrm{L}\mathord{<}}$ if necessary to obtain a term of the form $x ([x_{i_{k+1}}, \dots, x_{i_m}] [x_{i_1}, \dots, x_{i_k}])$, where
$i_1 < \dots < i_k$, $i_{k+1} < \dots < i_m$ and $i_1 < i_{k+1}$.
Using the identities $x(yz) \approx x(zy)$ and $(xy)z \approx (xz)y$, we obtain
\begin{align*}
xs
&
\approx \langle x, [x_{i_{k+1}}, \dots, x_{i_m}] [x_{i_1}, \dots, x_{i_k}] \rangle
= \langle x, ([x_{i_{k+1}}, \dots, x_{i_{m-1}}] x_{i_m}) [x_{i_1}, \dots, x_{i_k}] \rangle
\\ &
\approx \langle x, ( [x_{i_{k+1}}, \dots, x_{i_{m-1}}] [x_{i_1}, \dots, x_{i_k}] ) x_{i_m} \rangle
\approx \langle x, x_{i_m} ( [x_{i_{k+1}}, \dots, x_{i_{m-1}}] [x_{i_1}, \dots, x_{i_k}] ) \rangle
\\ &
= \langle x, x_{i_m}, [x_{i_{k+1}}, \dots, x_{i_{m-1}}] [x_{i_1}, \dots, x_{i_k}] \rangle
\approx \langle x, x_{i_m}, ( [x_{i_{k+1}}, \dots, x_{i_{m-2}}] [x_{i_1}, \dots, x_{i_k}] ) x_{i_{m-1}} \rangle
\\ &
\approx \langle x, x_{i_m}, x_{i_{m-1}}, [x_{i_{k+1}}, \dots, x_{i_{m-2}}] [x_{i_1}, \dots, x_{i_k}] \rangle
\approx \cdots
\approx \langle x, x_{i_m}, x_{i_{m-1}}, \dots, x_{i_{k+2}}, x_{i_{k+1}}, [x_{i_1}, \dots, x_{i_k}] \rangle
\\ &
\approx \langle x, x_{i_m}, x_{i_{m-1}}, \dots, x_{i_{k+2}}, [x_{i_1}, \dots, x_{i_k}] x_{i_{k+1}} \rangle
= \langle x, x_{i_m}, x_{i_{m-1}}, \dots, x_{i_{k+2}}, [x_{i_1}, \dots, x_{i_k}, x_{i_{k+1}}] \rangle
\\ &
\approx \langle x, x_{i_m}, x_{i_{m-1}}, \dots, x_{i_{k+3}}, [x_{i_1}, \dots, x_{i_k}, x_{i_{k+1}}, x_{i_{k+2}}] \rangle
\approx \dots
\approx \langle x, [x_{i_1}, \dots, x_{i_k}, x_{i_{k+1}}, \dots, x_{i_m}] \rangle.
\end{align*}
Since $i_1$ is the smallest of the indices $i_1, \dots, i_m$, we can then apply the identity $(xy)z \approx (xz)y$ and part \ref{lem:t:order-ti} to sort the variables in the subterm $[x_{i_1}, \dots, x_{i_k}, x_{i_{k+1}}, \dots, x_{i_m}]$ in the increasing order of indices, and we obtain $x (x^{\mathrm{L}\mathord{<}})$, as desired.
\end{proof}

\section{Polynomial upper bounds}\label{sec:poly}

In this section, we establish some polynomial upper bounds for the ac-spectra of groupoids belonging to certain varieties of groupoids; in contrast, their associative spectra all have constant upper bounds.

For our first variety of groupoids, we can actually determine their associative spectrum and ac-spectrum.

\begin{proposition}\label{prop:275}
A groupoid $(G,*)$ with at least two elements satisfying the identity $xy\approx x$ must have $\spa{*}=1$ and $\spac{*}=n$ for $n\ge1$.
In particular, the above two equalities hold for the $2$-element groupoid $(\{0,1\},*)$ defined by $x*y:= x$ for all $x,y\in\{0,1\}$ and the $3$-element groupoids $\SC275$ and $\SC2029$.
\begin{center}
\begin{tabular}{ccc}
\begin{tabular}{c|ccc}
$*$ & $0$ & $0$ & $0$ \\
\hline
$0$ & $0$ & $0$ & $0$ \\
$1$ & $1$ & $1$ & $1$ \\
$2$ & $2$ & $2$ & $2$ 
\end{tabular}
&
\begin{tabular}{c|ccc}
$*$ & $0$ & $1$ & $2$ \\
\hline
$0$ & $0$ & $1$ & $2$ \\
$1$ & $0$ & $1$ & $2$ \\
$2$ & $0$ & $1$ & $2$ 
\end{tabular}
\\
$\SC275$ & $\SC2029$ 
\end{tabular}
\end{center}
\end{proposition}

\begin{proof}
If $(G,*)$ is a groupoid with at least two elements satisfying the identity $xy\approx x$, then $\spa{*}=1$ and $\spac{*}=n$ for all $n\ge1$ since the $n$-ary operation $t^*$ induced by every term $t\in \F_n$ is determined by the leftmost variable of $t$ and distinct variables induce distinct operations.

In earlier work~\cite[Example~4.1.2]{AC-Spectrum}, we showed that the $2$-element groupoid $(\{0,1\},*)$ with $x*y:= x$ for all $x,y\in\{0,1\}$ has $\spa{*}=1$ and $\spac{*}=n$ for $n\ge1$.
One can check that $\SC275$ satisfies the identity $xy \approx x$ and that $\SC2029$ is anti-isomorphic to $\SC275$.
Thus their associative spectrum and ac-spectrum are also given as above.
\end{proof}

The upper bounds in the next result are achieved by the $3$-element groupoids $\SC7$ and $\SC28$, which are anti-isomorphic to $\SC4$ (by swapping $1$ and $2$) and $\SC5$, respectively.
\begin{center}
\begin{tabular}{ccccccc}
\begin{tabular}{c|ccc}
$*$ & $0$ & $1$ & $2$ \\
\hline
$0$ & $0$ & $0$ & $0$ \\
$1$ & $0$ & $0$ & $0$ \\
$2$ & $0$ & $1$ & $0$ 
\end{tabular}
&
\begin{tabular}{c|ccc}
$*$ & $0$ & $1$ & $2$ \\
\hline
$0$ & $0$ & $0$ & $0$ \\
$1$ & $0$ & $0$ & $0$ \\
$2$ & $0$ & $1$ & $1$ 
\end{tabular}
&
\begin{tabular}{c|ccc}
$*$ & $0$ & $1$ & $2$ \\
\hline
$0$ & $0$ & $0$ & $0$ \\
$1$ & $0$ & $0$ & $0$ \\
$2$ & $0$ & $2$ & $0$ 
\end{tabular}
&
\begin{tabular}{c|ccc}
$*$ & $0$ & $1$ & $2$ \\
\hline
$0$ & $0$ & $0$ & $0$ \\
$1$ & $0$ & $0$ & $1$ \\
$2$ & $0$ & $0$ & $1$ 
\end{tabular}
\\
$\SC4$ & $\SC5$ & $\SC7$ & $\SC28$
\end{tabular}
\end{center}

\begin{proposition}\label{prop:7}
A groupoid $(G,*)$ satisfying the identities below must have $\spa{*}\le 2$ and $\spac{*} \le n+1$ for $n=3,4,\ldots$, where the first inequality holds as an equality if so does the second and both equalities hold for $\SC7$ and $\SC28$.
\[ \mathrm{(i)}\ (xy)z\approx (xz)y, \quad
\mathrm{(ii)}\ x(yz)\approx x(zy) \approx y(xz), \quad
\mathrm{(iii)}\ w(x(yz)) \approx w((xy)z) \]
\end{proposition}

\begin{proof}
Let $t$ be an arbitrary term in $\F_n$ with leftmost decomposition $t = [t_0, t_1, \ldots, t_m]$, where $t_0=x_a$ for some $a\in\{1,2,\ldots,n\}$.
By Lemma~\ref{lem:t}, we may assume that $t_i = t_i^{\mathrm{R}\mathord{<}}$ for all $i \in \{1, \dots, m\}$ and $|t_1| \leq \dots \leq |t_m|$.
We then distinguish two cases below.
\begin{itemize}
\item
If $|t_m|>1$, then $t_m = \langle x_{b_1}, \dots, x_{b_\ell} \rangle$ and we can apply (ii) to swap the leftmost variable of $t_m$ and $[x_a, t_1, \ldots, t_{m-1}]$. 
The resulting term $x_{b_1} \langle [x_a, t_1, \ldots, t_{m-1}], x_{b_2}, \dots, x_{b_\ell} \rangle$ can be transformed to $\langle x_{b_1}, x_1, \dots, x_{b_1 - 1}, x_{b_1 + 1}, \dots, x_n \rangle$ by Lemma~\ref{lem:t}.
Then we can apply (ii) to swap $x_{b_1}$ with $x_1$, and finally we can turn the term into $\langle x_1, \ldots, x_n \rangle$.
\item
If $|t_m|=1$, then $t=[x_a, x_{b_1}, \ldots, x_{b_{n-1}}]$, and we can apply (i) to make sure $b_1<\cdots<b_{n-1}$.
\end{itemize}
It follows that $\spac{*}\le n+1$ since there are $n$ possibilities for $a$ in the second case.
If the variables $x_1, \ldots, x_n$ are ordered increasingly in $t$, then we must have $a=1$ in the second case. 
Thus $\spa{*}\le 2$.
If $\spac{*}=n+1$, then the two cases above cannot yield identical $n$-ary operations on $(G,*)$, and thus $\spa{*}=2$.

Now we determine $\spa{*}$ and $\spac{*}$ for $\SC7$.
It is routine to check that $\SC7$ satisfies the identities (i), (ii), and (iii).
Let $t$ be an arbitrary term in $\F_n$.
We may assume that $t= \langle x_1, \ldots, x_n\rangle$ or $t = [x_a, x_{b_1}, \ldots, x_{b_{n-1}}]$ with $b_1<\cdots<b_{n-1}$ by the above argument.
For the former, we have $h(t) = 0$ for all $h: X_n\to\{0,1,2\}$.
For the latter, we have $h(t) = 2$ if $h(x_a)=2$ and $h(x_{b_1})=\cdots=h(x_{b_{n-1}})=1$ or $h(t)=0$ otherwise.
Therefore $\spac{*}=n+1$, which implies $\spa{*}=2$. 

In a similar way, we can determine $\spa{*}$ and $\spac{*}$ for $\SC28$, which also satisfies the identities (i), (ii), and (iii).
If $t= \langle x_1, \ldots, x_n\rangle$, then $h(t) = 0$ for all $h:X_n\to\{0,1,2\}$.
If $t = [x_a, x_{b_1}, \ldots, b_{n-1}]$, then $h(t) = 1$ if $h(x_a)\in\{1,2\}$ and $h(x_{b_i})=2$ for $i=1,\ldots,n-1$, or $h(t)=0$ otherwise.
It follows that $\spac{*}=n+1$, which implies $\spa{*}=2$.
\end{proof}

The upper bounds in the next result are very close to but not the same as those in Proposition~\ref{prop:7}.

\begin{proposition}\label{prop:405}
A groupoid $(G,*)$ satisfying the identities below must have upper bounds $\spa{*}\le 2$ and $\spa{*} \le 3$ for $n=3$ and $\spa{*} \le 3$ and $\spac{*} \le n+1$ for $n=4,5,\ldots$.
\[ \mathrm{(i)}\ xy \approx yx, \quad
\mathrm{(ii)}\ w(x(yz)) \approx w((xy)z), \quad
\mathrm{(iii)}\ (v(wx))(yz) \approx (vw)(x(yz)) \]
If $\spac{*}$ reaches its upper bound, so does $\spa{*}$, and both upper bounds are reached by $\SC405$ (see Table~\ref{tab:3-element groupoids}).
\end{proposition}

\begin{proof}
Let $t$ be an arbitrary term in $\F_n$ with leftmost decomposition $t = [t_0, t_1, \ldots, t_m]$, where $|t_0|=1$.
By Lemma~\ref{lem:t}, we can assume that $t_i = t_i^{\mathrm{R}\mathord{<}}$ for all $i \in \{1, \dots, m\}$.
We can use (i) to swap $[t_0, t_1, \ldots, t_{m-1}]$ and $t_m$.
The resulting term $t_m [t_0, t_1, \ldots, t_{m-1}]$ can be transformed to
\[ \langle x_{i_1}, \ldots, x_{i_k} \rangle \langle x_{i_{k+1}}, \ldots, x_{i_n} \rangle, \]
by Lemma~~\ref{lem:t} and (i) again, where $\{x_{i_1}, \dots, x_{i_k}\} = \var([t_0, t_1, \dots, t_{m-1}])$ and $\{x_{i_{k+1}}, \ldots, x_{i_n}\} = \var(t_m)$ with $i_1 < \dots < i_k$ and $i_{k+1} < \dots < i_n$.
Note that $i_j = j$ for $j = 1, \ldots, n$ if $t \in \B_n$.
If $k=2$, then we can show that 
\[ (\langle x_{i_1}, x_{i_2} \rangle \langle x_{i_3}, \ldots, x_{i_n} \rangle)^* = (\langle x_1, x_2 \rangle \langle x_3, \ldots, x_n \rangle)^*. \]
We have either $i_1 = 1$ or $i_3 = 1$.
If $i_3 = 1$, then we can do the following transformations to make the leftmost index $1$.
\begin{multline*}
\langle i_1, i_2 \rangle \langle i_3, \ldots, i_n\rangle 
\xrightarrow{\text{(iii)}} 
\langle i_1, i_2, i_3 \rangle \langle i_4, \ldots, i_n\rangle 
\xrightarrow{\text{(i)}} 
\langle i_4, \ldots, i_n\rangle \langle i_1, i_2, i_3 \rangle
\\
\xrightarrow{\text{Lemma~\ref{lem:t}}} 
\langle i_4, \ldots, i_n\rangle \langle i_3, i_1, i_2 \rangle
\xrightarrow{\text{(i)}} 
\langle i_3, i_1, i_2 \rangle \langle i_4, \ldots, i_n \rangle
\xrightarrow{\text{(iii)}} 
\langle i_3, i_1 \rangle \langle i_2, i_4, \ldots, i_n \rangle 
\end{multline*}
Here we drop $x$ for ease of notation and represent an application of an identity by an arrow with the label of the identity above it.
Similarly, we can make the second leftmost index $2$ and then make the rest $3, \ldots, n$.
If $3\le k \le n-2$, then we have 
\[
\langle i_1, \ldots, i_k \rangle \langle i_{k+1}, \ldots, i_n\rangle 
\xrightarrow{\mathrm{(iii)}} 
\langle i_1, i_2 \rangle ( \langle i_3, \dots, i_k \rangle \langle i_{k+1}, \dots, i_n \rangle )
\xrightarrow{\text{Lemma~\ref{lem:t}}} 
\langle i_1, i_{2} \rangle \langle i_{3}, \ldots, i_n\rangle.
\]
Here the application of (iii) uses $v=x_{i_1}$, $w=x_{i_2}$, $x=\langle x_{i_3}, \ldots, x_{i_k} \rangle$, $y=x_{i_{k+1}}$, and $z=\langle x_{i_{k+2}}, \ldots, x_{i_n} \rangle$.
Thus $t$ induces the same $n$-ary operation on $(G,*)$ as one of the following ``standard'' terms.
\[
x_{i_1} \langle x_{i_2}, \ldots, x_{i_n} \rangle, \quad
\langle x_{i_1}, \ldots, x_{i_{n-1}} \rangle x_{i_n}, \quad
\langle x_1, x_2 \rangle \langle x_3, \ldots, x_n \rangle
\]

The first standard term is determined by $i_1$ since $i_2<\cdots<i_n$, and the second is determined by $i_n$ since $i_1<\cdots<i_{n-1}$.
Moreover, $x_{i_1} \langle x_{i_2}, \ldots, x_{i_n} \rangle$ and $\langle x_{i_1}, \ldots, x_{i_{n-1}} \rangle x_{i_n}$ induce the same $n$-ary operation on $(G,*)$ if $i_1=i_n$ by (i).
Thus there are $n$ possibilities in total for the first two standard terms.
On the other hand, the last standard term $\langle x_1, x_2 \rangle \langle x_3, \ldots, x_n \rangle$ does not occur when $n=3$.
Thus $\spac{*}\le 3$ when $n=3$ and $\spac{*} \le n+1$ for $n\ge4$.

If $t\in\B_n$ is a bracketing of $x_1, \ldots, x_n$, then by the above argument, it induces the same $n$-ary operation on $(G,*)$ as one of  $x_1 \langle x_2, \ldots, x_n \rangle$, $\langle x_1, \ldots, x_{n-1} \rangle x_n$, or $\langle x_1, x_2 \rangle \langle x_3, \ldots, x_n \rangle$.
Thus $\spa{*}\le 2$ for $n=3$ and $\spa{*}\le 3$ for $n\ge4$.
It is easy to see that the equality holds in the upper bound for $\spa{*}$ when the equality holds in the upper bound for $\spac{*}$.

Now we consider $\SC405$.
Write an arbitrary term $t\in \F_n$ as $t=(t_L)(t_R)$, where $t_L$ and $t_R$ are linear terms.
Also view $t$ as a bracketing of $x_{i_1}, \ldots, x_{i_n}$.
We distinguish the following cases on $|t_L|$ and $|t_R|$.
\begin{itemize}
\item[(i)]
If $|t_L|=1<|t_R|$ then $t_R^*$ evaluates to $0$ or $1$, so $t^*(a_1, \dots, a_n) = \lfloor a_{i_1}/2 \rfloor$.
\item[(ii)]
If $|t_L|>1=|t_R|$ then $t^*(a_1, \dots, a_n) = \lfloor a_{i_n}/2 \rfloor$ for the same reason as above.
\item[(iii)]
If $|t_L|\ge2$ and $|t_R|\ge2$ then $t_L^*$ and $t_R^*$ both evaluate to $0$ or $1$, so $t^*$ is always zero.
\end{itemize}

For $n=3$ we must have (i) or (ii), so $t^*(a_1, a_2, a_3) = \lfloor a_i / 2 \rfloor$, where $i$ varies in $\{1,2,3\}$. Thus $\spac{*} = 3$ for $n=3$.
For $n\ge4$, we have $t^*(a_1, \dots, a_n) = \lfloor a_i / 2 \rfloor$, where $i$ varies in $\{1,2,\ldots,n\}$, or $t^* = 0$.
Thus $\spac{*} = n+1$ for $n\ge4$.
\end{proof}

The next result involves the $3$-element groupoid $\SC189$, which is anti-isomorphic to $\SC170$.
\begin{center}
\begin{tabular}{ccccccc}
\begin{tabular}{c|ccc}
* & 0 & 1 & 2 \\
\hline
0 & 0 & 0 & 0 \\
1 & 0 & 2 & 1 \\
2 & 0 & 2 & 1 
\end{tabular}
&
\begin{tabular}{c|ccc}
* & 0 & 1 & 2 \\
\hline
0 & 0 & 0 & 0 \\
1 & 0 & 2 & 2 \\
2 & 0 & 1 & 1 
\end{tabular}
\\
$\SC170$ & $\SC189$
\end{tabular}
\end{center}

\begin{proposition}\label{prop:189}
A groupoid $(G,*)$ satisfying the identities below must have $\spa{*} \le 2$ and $\spac{*}\le 2n$ for $n=3,4,\ldots$, where the first inequality holds as an equality whenever the second does and both hold for the $2$-element groupoid $(\{0,1\},*)$ defined by $x*y:= x+1 \pmod 2$ for all $x,y\in\{0,1\}$ and the $3$-element groupoids $\SC170$ and $\SC189$.
\[ \mathrm{(i)}\ x(yz) \approx x(zy), \quad
\mathrm{(ii)}\ (xy)z \approx (xz)y, \quad
\mathrm{(iii)}\ w(x(yz)) \approx w((xy)z), \]
\[ \mathrm{(iv)}\ (wx)(yz) \approx (w(xy))z, \quad 
\mathrm{(v)}\ w(x(yz)) \approx ((wx)y)z \]
\end{proposition}

\begin{proof}
Let $t$ be an arbitrary term in $\F_n$ with leftmost decomposition $t = [t_0, t_1, \ldots, t_m]$, where $|t_0| = 1$.
By (i) and (iii) and Lemma~\ref{lem:t}, we may assume that $t_i = t_i^{\mathrm{R}\mathord{<}}$ for all $i \in \{1, \dots, m\}$.
By (v), we may assume $m \leq 2$.
If $m=1$ then $t^* = \langle x_{i_1}, \ldots, x_{i_n} \rangle^*$ with $i_2<\cdots<i_n$.
If $m=2$ then we can further use (iv) to obtain $t^* = [x_{i_1}, x_{i_2}, \langle x_{i_3}, \ldots, x_{i_n} \rangle]^*$ and make sure $i_2<\cdots<i_n$ by (ii). 
Thus $\spac{*}\le 2n$.

If $t$ is a bracketing of $x_1, \ldots, x_n$, then $t^* = \langle x_1, \ldots, x_n \rangle^*$ or $t^* = [x_1, x_2, \langle x_3, \ldots, x_n \rangle]^*$ by a similar argument.
Thus $\spa{*} \le 2$, and the equality must hold if $\spac{*}=2n$.

It is routine to check that the $2$-element groupoid $(\{0,1\},*)$ defined by $x*y:= x+1 \pmod 2$ for all $x,y\in\{0,1\}$ satisfies the identities (i)--(v).
It has $\spa{*}=2$ for $n\ge2$ by Cs\'{a}k\'{a}ny and Waldhauser~\cite[\S4.1]{AssociativeSpectra1} and $\spac{*}=2n$ for $n\ge3$ by our earlier work~\cite[Example~4.1.2]{AC-Spectrum}.
It is easy to see that $\SC189$ is obtained from this $2$-element groupoid by adding an absorbing element; hence the term operations behave in essentially the same ways in both groupoids.
\end{proof}

Our next result is similar to Proposition~\ref{prop:189}, and we will use leaf depths to generalize them in Section~\ref{sec:depth}.
The result here involves two anti-isomorphic groupoids:
\begin{center}
\begin{tabular}{ccc}
\begin{tabular}{c|ccc}
$*$ & $0$ & $1$ & $2$ \\
\hline
$0$ & $1$ & $1$ & $1$ \\
$1$ & $2$ & $2$ & $2$ \\
$2$ & $0$ & $0$ & $0$ 
\end{tabular}
&
\begin{tabular}{c|ccc}
$*$ & $0$ & $1$ & $2$ \\
\hline
$0$ & $1$ & $2$ & $0$ \\
$1$ & $1$ & $2$ & $0$ \\
$2$ & $1$ & $2$ & $0$ 
\end{tabular}
\\
$\SC3242$ & $\SC3302$
\end{tabular}
\end{center}

\begin{proposition}\label{prop:3242}
A groupoid $(G,*)$ satisfying the identities below must have
\[ \spa{*} \le 
\begin{cases} 
1 & n=1,2 \\
2 & n=3 \\
3 & n=4,5,\ldots
\end{cases} 
\quad \text{and} \quad
\spac{*} \le 
\begin{cases} 
n & n=1,2 \\
2n & n=3 \\
3n & n=4,5,\ldots
\end{cases} \]
where the first inequality holds as an equality if so does the second and both hold for $\SC3242$ and the anti-isomorphic $\SC3302$. 
\[
\mathrm{(i)}\ x(yz) \approx x(zy), \quad
\mathrm{(ii)}\ w(x(yz)) \approx w((xy)z), \quad
\mathrm{(iii)}\ (xy)z \approx (xz)y,
\]
\[
\mathrm{(iv)}\ (wx)(yz) \approx (wy)(xz), \quad
\mathrm{(v)}\ (((vw)x)y)z \approx v(w(x(yz))) \]
\end{proposition}

\begin{proof}
The result is trivial when $n=1,2$; assume $n\ge3$ below.
Let $t$ be an arbitrary term in $\F_n$ with leftmost decomposition $t = [t_0, t_1, \ldots, t_m]$, where $|t_0| = 1$.
By (i), (ii), and Lemma~\ref{lem:t}, we may assume that $t_i = t_i^{\mathrm{R}\mathord{<}}$ for all $i \in \{1, \dots, m\}$.
By (iii), we may assume that $|t_1| \leq \dots \leq |t_m|$.
If $|t_i| > 1$ for some $i \in \{1, \dots, m-1\}$, then we apply (iv) to make sure $|t_i|=1$.
Thus we may assume that $|t_1| = \dots = |t_{m-1}| = 1$.
Therefore, $t$ induces the same $n$-ary operation on $(G,*)$ as $[x_{i_1}, \ldots, x_{i_k}, \langle x_{i_{k+1}}, \ldots, x_{i_n} \rangle]$, where
we may further assume that $x_{i_2}<\cdots<x_{i_n}$ by (iii) and (iv) and that $k\in\{1,2,3\}$ by (v).
It follows that $\spac{*} \le 2n$ for $n=3$ (in this case $k\in\{1,2\}$) and $\spac{*}\le 3n$ for $n=4,5,\ldots$.

If $t\in\B_n$ is a bracketing of $x_1x_2\cdots x_n$, then we must have $i_1=1$ since the above argument does not alter the leftmost variable. 
Thus $\spa{*} \le 2$ for $n=3$ and $\spac{*}\le 3$ for $n=4,5,\ldots$.
It is clear that if the upper bound of $\spac{*}$ is reached, so is the upper bound of $\spa{*}$.

For $\SC3242$, we have $t^*(a_1, \dots, a_n) = ( a_{i_1} + d ) \bmod 3$ whenever the binary tree corresponding to $t \in \F_n$ has leftmost leaf $i_1$ of left depth $d$.
The number of possibilities for $i_1$ is $n$, and the number of possibilities for $d\pmod 3$ is $1$ when $n\in\{1,2\}$, $2$ when $n=3$, and $3$ when $n=4,5,\ldots$.
The proof is now complete.
\end{proof}

We next present a family of groupoids whose associative spectrum and ac-spectrum are bounded above by $1,1,2,4,4,4,4,\ldots$ and $1,2,9,32,50,72,98,\ldots$ and show that both upper bounds are reached by the $3$-element groupoid $\SC3162$, which is anti-isomorphic to $\SC2467$.
\begin{center}
\begin{tabular}{ccccccc}
\begin{tabular}{c|ccc}
$*$ & $0$ & $1$ & $2$ \\
\hline
$0$ & $1$ & $0$ & $0$ \\
$1$ & $1$ & $0$ & $0$ \\
$2$ & $1$ & $0$ & $1$ 
\end{tabular}
&
\begin{tabular}{c|ccc}
$*$ & $0$ & $1$ & $2$ \\
\hline
$0$ & $1$ & $1$ & $1$ \\
$1$ & $0$ & $0$ & $0$ \\
$2$ & $0$ & $0$ & $1$ 
\end{tabular}
\\
$\SC2467$ & $\SC3162$
\end{tabular}
\end{center}

\begin{theorem}\label{thm:3162}
A groupoid $(G,*)$ satisfying the identities below must have $\spa{*}\le 2$ and $\spac{*}\le n^2$ for $n=3$ and $\spa{*}\le 4$ and $\spac{*}\le 2n^2$ for $n=4,5,\ldots$, where the upper bound for $\spa{*}$ is reached if the upper bound for $\spac{*}$ is reached and both upper bounds are reached by $\SC3162$ and the anti-isomorphic $\SC2467$.
\[ \mathrm{(i)}\ x(yz) \approx x(zy), \quad
\mathrm{(ii)}\ w(x(yz)) \approx w((xy)z), \quad
\mathrm{(iii)}\ ((wx)y)z \approx ((wx)z)y, \]
\[ \mathrm{(iv)}\ (w(xy))z \approx (w(xz))y, \quad
\mathrm{(v)}\ v(w(x(yz))) \approx ((v(wx))y)z, \quad
\mathrm{(vi)}\ (vw)(x(yz)) \approx (((vw)x)y)z \]
\end{theorem}

\begin{proof}
Let $t$ be an arbitrary term in $\F_n$ with leftmost decomposition $t = [x_a, t_1, t_2, \ldots, t_m]$. 
By (i), (iii), and Lemma~\ref{lem:t}, we may assume that $t_i = t_i^{\mathrm{L}\mathord{<}}$ for all $i \in \{1, \dots, m\}$.
By (vi), we may assume that $m \leq 3$.
Consequently, $t$ induces the same $n$-ary operation on $(G,*)$ as one of the following four types of standard terms.

\vskip5pt\noindent\textsf{Type 1}: $m=1$.
Then $t^* = (x_a [x_{b_1},\ldots,x_{b_{n-1}}])^*$, where $b_1<\cdots<b_{n-1}$.

\vskip5pt\noindent\textsf{Type 2}: $m=2$ and $|t_1|=1$.
Then $t^* = ([x_a, x_b, [x_{c_1},\ldots,x_{c_{n-2}}]])^*$, where $c_1<\cdots<c_{n-2}$.

\vskip5pt\noindent\textsf{Type 3}: $m=2$ and $|t_1|\ge2$.
Then $t^* = ([x_a, [x_{b_1}, \ldots,x_{b_{n-2}}], x_{b_{n-1}}])^*$, where $b_1<\cdots<b_{n-1}$, thanks to (iv).

\vskip5pt\noindent\textsf{Type 4}: $m=3$.
We may assume that $|t_1|=1$ by the identity (v) and that $|t_2|\ge |t_3|$ by the identity (iii).
If $|t_2|\ge |t_3|>1$, then we can write $t_2 = t'_2 x$ for any variable $x \in \var(t_2)$ and switch $x$ with $t_3$ by (iv).
Thus we may also assume $|t_3|=1$.
It follows that $t^* = ([x_a, x_b, [x_{c_1}, \ldots,x_{c_{n-3}}], x_{c_{n-2}}])^*$, where $c_1<\cdots<c_{n-2}$.

\vskip5pt
Summing up the possibilities for the above four types of standard terms, we obtain that 
\[ \spac{*}\le
\begin{cases}
n+n(n-1) = n^2 & \text{if } n=3 \\
n+n(n-1)+n+n(n-1)=2n^2 & \text{if } n=4,5,\ldots.
\end{cases} \]
If $t\in\B_n$ is a bracketing of $x_1 x_2 \cdots x_n$, then there is only one possibility in each of the above four (two when $n=3$) cases.
This shows that $\spa{*}\le 2$ for $n=3$ and $\spa{*}\le 4$ for $n=4,5,\ldots$.
If $\spac{*}=2n^2$ then the above four cases must induce distinct terms on $(G,*)$, and thus $\spa{*}=4$.

It is routine to check that $\SC3162$ satisfies the identities (i)--(vi).
It remains to show that any two distinct standard terms $t$ and $t'$ in $\F_n$ must induce distinct $n$-ary operations on $\SC3162$.
Assume that $t$ is one of the following, where $b_1<\cdots<b_{n-1}$ and $c_1<\cdots<c_{n-2}$.
\[ x_a [x_{b_1},\ldots,x_{b_{n-1}}],\
[x_a, x_b, [x_{c_1},\ldots,x_{c_{n-2}}]],\ 
[x_a, [x_{b_1}, \ldots,x_{b_{n-2}}], x_{b_{n-1}}], \ 
[x_a, x_b, [x_{c_1}, \ldots,x_{c_{n-3}}], x_{c_{n-2}}]
\]
Similarly, assume that $t'$ is one of the following, where $b'_1<\cdots<b'_{n-1}$ and $c'_1<\cdots<c'_{n-2}$.
\[ x_{a'} [x_{b'_1},\ldots,x_{b'_{n-1}}],\
[x_{a'}, x_{b'}, [x_{c'_1},\ldots,x_{c'_{n-2}}]],\ 
[x_{a'}, [x_{b'_1}, \ldots,x_{b'_{n-2}}], x_{b'_{n-1}}], \ 
[x_{a'}, x_{b'}, [x_{c'_1}, \ldots,x_{c'_{n-3}}], x_{c'_{n-2}}]
\]
It is clear that $[0,s_1,\ldots,s_\ell]$ gives $0$ if $\ell$ is even or $1$ if $\ell$ is odd, no matter what $s_1,\ldots,s_\ell$ are.
Therefore, we only need to consider the following cases.

\vskip5pt\noindent\textsf{Case 1}: $t = x_a [x_{b_1},\ldots,x_{b_{n-1}}]$ and $t' = x_{a'} [x_{b'_1},\ldots,x_{b'_{n-1}}]$, where $a\ne a'$.
We have $h(t) = 0 \ne 1 = h(t')$, where $h(x_a):=1$ and $h(x):=0$ for all $x\ne x_a$.

\vskip5pt\noindent\textsf{Case 2}: $t = x_a [x_{b_1},\ldots,x_{b_{n-1}}]$ and $t' = [x_{a'}, x_{b'}, [x_{c'_1}, \ldots,x_{c'_{n-3}}], x_{c'_{n-2}}]$.
We have $h(t) = 0 \ne 1 = h(t')$, where $h(x_a)=h(x_{a'})=h(x_{b'}):=2$ and $h(x):=0$ for all $x\notin\{x_a,x_{a'},x_{b'}\}$.
Here $a$ may coincide with $a'$ or $b'$.

\vskip5pt\noindent\textsf{Case 3}: $t = [x_{a}, x_{b}, [x_{c_1},\ldots,x_{c_{n-2}}]]$ and $t' = [x_{a'}, x_{b'}, [x_{c'_1},\ldots,x_{c'_{n-2}}]]$, where $(a,b)\ne(a',b')$.

If $a\ne a'$ then $h(t) = 0 \ne 1 = h(t')$, where $h(x_a):=0$ and $h(x):=1$ for all $x\ne x_a$.

If $a= a'$ then $b\ne b'$ and $h(t) = 0 \ne 1 = h(t')$, where $h(x_a)=h(x_b):=2$ and $h(x):=0$ for all $x\notin\{x_a,x_b\}$.

\vskip5pt\noindent\textsf{Case 4}: $t = [x_{a}, x_{b}, [x_{c_1},\ldots,x_{c_{n-2}}]]$ and $t' = [x_{a'}, [x_{b'_1}, \ldots,x_{b'_{n-2}}], x_{b'_{n-1}}]$.

If $a\ne a'$ then $h(t) = 0 \ne 1 = h(t')$, where $h(x_a):=0$ and $h(x):=1$ for all $x\ne x_a$.

If $a=a'$ then $h(t) = 0 \ne 1 = h(t')$, where $h(x_a)=h(x_b):=2$ and $h(x):=0$ for all $x\notin\{x_a,x_b\}$.

\vskip5pt\noindent\textsf{Case 5}: $t = [x_{a}, [x_{b_1}, \ldots,x_{b_{n-2}}], x_{b_{n-1}}]$ and $t' = [x_{a'}, [x_{b'_1}, \ldots,x_{b'_{n-2}}], x_{b'_{n-1}}]$, where $a\ne a'$.
We have $h(t) = 0 \ne 1 = h(t')$, where $h(x_a):=0$ and $h(x):=1$ for all $x\ne x_a$.

\vskip5pt\noindent\textsf{Case 6}: $t = [x_a, x_b, [x_{c_1}, \ldots,x_{c_{n-3}}], x_{c_{n-2}}]$ and $t' = [x_{a'}, x_{b'}, [x_{c'_1}, \ldots,x_{c'_{n-3}}], x_{c'_{n-2}}]$, where $(a,b)\ne (a',b')$.

If $a\ne a'$ then $h(t) = 1 \ne 0 = h(t')$, where $h(x_a):=0$ and $h(x):=1$ for all $x\ne x_a$.

If $a= a'$ then $b\ne b'$ and $h(t) = 1 \ne 0 = h(t')$, where $h(x_a)=h(x_b):=2$ and $h(x):=0$ for all $x\notin\{x_a,x_b\}$.

\vskip5pt
The proof is now complete.
\end{proof}

\section{Exponential upper bounds}\label{sec:expo}

In this section, we establish some exponential upper bounds for the ac-spectra for a few varieties of groupoids; the respective associative spectra may have linear or exponential upper bounds.

\begin{proposition}\label{prop:1066}
Every groupoid $(G,*)$ satisfying the identities below must have $\spa{*}\le n-1$ and $\spac{*}\le 2^{n-1}-1$ for $n=2,3,\ldots$, where the first inequality holds as an equality whenever the second does and both equalities hold for $\SC1066$ (see Table~\ref{tab:3-element groupoids}).
\[ \mathrm{(i)}\ xy \approx yx, \quad
\mathrm{(ii)}\ w(x(yz)) \approx w((xy)z) \]
\end{proposition}

\begin{proof}
Let $t$ be an arbitrary term in $\F_n$ with leftmost decomposition $t = [x_a, t_1, t_2, \ldots, t_m]$.
By Lemma~\ref{lem:t}, we may assume that $t_i = t_i^{\mathrm{L}\mathord{<}}$ for all $i \in \{1, \dots, m\}$.
Next, we use (i) to swap $[x_a, t_1, \ldots, t_{m-1}]$ and $t_m$.
Then we transform $[x_a, t_1, \ldots, t_{m-1}]$ to a leftmost bracketing again by Lemma~\ref{lem:t}.
It follows that $t$ induces the same $n$-ary operation on $(G,*)$ as $[x_{j_1}, \ldots, x_{j_k}] [x_{j_{k+1}}, \ldots, x_{j_n}]$, where $\{x_{j_1}, \ldots, x_{j_k}\} = \var(t_m)$ and $\{x_{j_{k+1}},\ldots, x_{j_n}\} = X_n\setminus \var(t_m)$.
The order of the elements of either sets of variables does not affect $t^*$ by the above, nor does the order of the two sets by (i).
Thus $\spac{*}$ is bounded above by $(2^n-2)/2 = 2^{n-1}-1$, the number of partitions of $\{1,\ldots,n\}$ into two unordered nonempty blocks.

Restricting the above argument to bracketings of $x_1 x_2 \cdots x_n$ in $\B_n$ instead of full linear terms in $\F_n$, we have the variables in $\var(t_m)$ indexed by larger numbers than the other variables.
Thus the partitions of $\{1,\ldots,n\}$ associated with these bracketings have two blocks $\{1,\ldots, k\}$ and $\{k+1,\ldots,n\}$ for some $k\in\{1,\ldots,n-1\}$.
It follows that $\spa{*}\le n-1$.

If $\spac{*}=2^{n-1}-1$, then distinct partitions of $\{1,\ldots,n\}$ into two unordered nonempty blocks correspond to distinct $n$-ary operations on $(G,*)$, and we can restrict this to partitions with two blocks $\{1,\ldots, k\}$ and $\{k+1,\ldots,n\}$ to conclude that $\spa{*}=n-1$.

It remains to consider $\SC1066$.
Every full linear term $t\in \F_n$ can be written as $t=t_L t_R$. 
Let $h: X_n\to\{0,1,2\}$ be an assignment.
We have that $h(t) = 1$ if and only if $h(t_L)=h(t_R)=2$ and that $h(t) = 0$ if and only if $h(t_1)\ne2$ and $h(t_2)\ne 2$. 
As observed by Cs\'{a}k\'{a}ny and Waldhauser~\cite{AssociativeSpectra1}, one can show by induction that $h(t) = 2$ if and only if $h$ assigns $2$ to an odd number of variables. 
Thus $h(t)$ is completely determined by how many variables in $t_1$ and $t_2$ take the value $2$.
In particular, if $s=[x_{i_1}, \ldots, x_{i_\ell}] [x_{i_{\ell+1}}, \ldots, x_{i_n}]$ and $t = [x_{j_1}, \ldots, x_{j_k}] [x_{j_{k+1}}, \ldots, x_{j_n}]$ with $1\in \{i_1, \ldots, i_\ell\} \cap \{j_1, \ldots, j_k\}$ and $i\in \{i_1, \ldots, i_\ell\}\setminus\{j_1, \ldots, j_k\}$, then $s^*\ne t^*$ since $h(s)=0\ne 1= h(t)$, where $h(x_1)=h(x_i):=2$ and $h(x):=0$ for all $x\notin\{x_1,x_i\}$.
This implies that $\spac{*}=2^{n-1}-1$, which in turn implies $\spa{*}=n-1$.
\end{proof}

We study another variety of groupoids, for which the associative spectra have the same upper bound $n-1$ as in Proposition~\ref{prop:1066} but the ac-spectra have a different upper bound $1,2,7,29,146, \ldots$~\cite[A185109]{OEIS}.
We show that both upper bounds are reached by $\SC367$, which is anti-isomorphic to $\SC10$.
\begin{center}
\begin{tabular}{ccccccc}
\begin{tabular}{c|ccc}
$*$ & $0$ & $0$ & $0$ \\
\hline
$0$ & $0$ & $0$ & $0$ \\
$1$ & $0$ & $0$ & $0$ \\
$2$ & $1$ & $0$ & $0$ 
\end{tabular}
&
\begin{tabular}{c|ccc}
$*$ & $0$ & $1$ & $2$ \\
\hline
$0$ & $0$ & $0$ & $1$ \\
$1$ & $0$ & $0$ & $0$ \\
$2$ & $0$ & $0$ & $0$ 
\end{tabular}
\\
$\SC10$ & $\SC367$ & 
\end{tabular}
\end{center}

\begin{proposition}\label{prop:367}
A groupoid $(G,*)$ must have $\spa{*}\le n-1$ for $n=2,3,\ldots$ and
\[ \spac{*} \le  n!+\sum_{k=0}^{n-3} n(n-1)\cdots(n-k+1) 
= n! + \sum_{k=0}^{n-3} k! \binom{n}{k} \]
for $n=1,2,\ldots$ if it satisfies the identities below, where the first inequality holds as an equality whenever the second does and both equalities hold for $\SC367$ and the anti-isomorphic $\SC10$.
\[ \mathrm{(i)}\ x(yz) \approx x(zy) \approx y(xz), \quad
\mathrm{(ii)}\ w(x(yz)) \approx w((xy)z) \]
\end{proposition}

\begin{proof}
We transform an arbitrary term $t\in F_n$, whose leftmost decomposition is $t=[t_0, t_1, \ldots, t_m]$ with $|t_0|=1$, to a ``standard'' term of the form $[ \langle x_{i_1}, \ldots, x_{i_\ell} \rangle, x_{i_{\ell+1}}, \ldots, x_{i_n} ]$, where $\ell \in \{0, 3, 4, \dots, n\}$ and $i_1 < \dots < i_\ell$.

If $|t_i|=1$ for all $i=1,\ldots,m$, then $t = [x_{i_1}, \ldots, x_{i_n}]$ is already a standard term with $\ell = 0$. 
Here $i_1, \ldots, i_n$ form a permutation of $1,\ldots,n$, and we have $n!$ possibilities in this case.

Suppose $|t_j| > 1$ for some $j$, where $j$ is as large as possible.
Then $|t_{j+1}| = \cdots = |t_m| = 1$.
We can transform $t_j$ to the rightmost bracketing of its variables in any prescribed order by (i), (ii), and Lemma~\ref{lem:t}, then switch its leftmost variable $x_{i_1}$ with $[t_0, t_1,\ldots, t_{j-1}]$ by (i), and use (i), (ii), and Lemma~\ref{lem:t} again to transform $t$ to the standard form $[\langle x_{i_1}, \ldots, x_{i_\ell} \rangle, x_{i_{\ell+1}}, \ldots, x_{i_n} ]$, where $i_1, \ldots, i_\ell$ can be in any prescribed order, say the increasing one.
There are $n(n-1)\cdots(\ell+1)$ possibilities for $i_{\ell+1}, \ldots, i_n$, and we must have $3\le \ell\le n$ since $|t_j|>1$.

Summing the numbers of possibilities in the above two cases with $k=n-\ell$ in the second case gives the desired upper bound for $\spac{*}$.
Restricting the above argument to bracketings of $x_1 x_2 \cdots x_n$ in $\B_n$, we obtain standard terms of the form $[\langle x_1,\ldots,x_\ell \rangle, x_{\ell+1}, \ldots, x_n]$ with $\ell \in\{0,3,4,\ldots,n\}$.
Thus $\spa{*}\le n-1$.
It is easy to see that if the upper bound for $\spac{*}$ is reached, so is the upper bound for $\spa{*}$.

It is clear that $\SC10$ is anti-isomorphic to $\SC367$.
The latter satisfies the identities (i) and (ii). 
It remains to show that $s^*\ne t^*$ whenever $s$ and $t$ are distinct standard terms in $\F_n$.
We may assume that $s= [\langle x_{i_1}, \ldots, x_{i_\ell} \rangle, x_{i_{\ell+1}}, \ldots, x_{i_n}]$ and $t= [ \langle x_{j_1}, \ldots, x_{j_m} \rangle, x_{j_{m+1}}, \ldots, x_{j_n}]$ for some $\ell,m\in\{0,3,4, \ldots, n\}$, where $i_1<\cdots<i_\ell$, $j_1<\cdots<j_m$, and $\ell\le m$.

First, assume that $i_k\ne j_k$ for some $k\in\{m+1,\ldots,n\}$.
Let $k$ be as large as possible.
We have $h(s) =0 \ne 1 = h(t)$ if $n-k$ is odd or $h(s) =1 \ne 0 = h(t)$ if $n-k$ is even, where $h(x_{i_{k}}) = \cdots = h(x_{i_n}) :=2$ and $h(x):=0$ for all $x\notin\{x_{i_{k}}, \ldots, x_{i_n}\}$.

Next, assume that $i_k=j_k$ for all $k=m+1, \ldots, n$.
This implies that $\ell<m$ (otherwise $s=t$).
We have $h(s) =0 \ne 1 = h(t)$ if $n-m$ is odd or $h(s) =1 \ne 0 = h(t)$ if $n-m$ is even, where $h(x_{i_m}) = \cdots = h(x_{i_n}) :=2$ and $h(x):=0$ for all $x\notin\{x_{i_m}, \ldots, x_{i_n}\}$.
\end{proof}

The upper bounds in the next result are reached by the $3$-element groupoid $\SC2302$, which can be viewed as subtraction on a finite field of three elements, or more generally, reached 
by the subtraction on any commutative group $(G,{+})$ of exponent greater than $2$ (cf. ~\cite[Example~7.1.4]{AC-Spectrum}).
It is clear $\SC2302$ is anti-isomorphic to $\SC2155$.
\begin{center}
\begin{tabular}{ccccccc}
\begin{tabular}{c|ccc}
$*$ & $0$ & $1$ & $2$ \\
\hline
$0$ & $0$ & $1$ & $2$ \\
$1$ & $2$ & $0$ & $1$ \\
$2$ & $1$ & $2$ & $0$ 
\end{tabular}
&
\begin{tabular}{c|ccc}
$*$ & $0$ & $1$ & $2$ \\
\hline
$0$ & $0$ & $2$ & $1$ \\
$1$ & $1$ & $0$ & $2$ \\
$2$ & $2$ & $1$ & $0$ 
\end{tabular}
&
\\
$\SC2155$ & $\SC2302$
\end{tabular}
\end{center}

\begin{proposition}\label{prop:subtration}
A groupoid $(G,*)$ satisfying the identities below must have $\spa{*} \le 2^{n-2}$ and $\spac{*} \le 2^n-2$ for $n=2,3,\ldots$, where the second inequality holds as an equality whenever the second does and both equalities hold for the subtraction operation $-$ on any commutative group $(G,{+})$ of exponent greater than $2$, in particular, for $\SC2302$ \textup{(}hence the anti-isomorphic $\SC2155$\textup{)}.
\[ \mathrm{(i)}\ (xy)z \approx (xz)y, \quad
\mathrm{(ii)}\ x(yz) \approx z(yx), \quad 
\mathrm{(iii)}\ w(x(yz)) \approx (w(xy))z. \]
\end{proposition}

\begin{proof}
Let $t$ be an arbitrary term in $\F_n$.
We show by induction on $|t|$ that $t$ can be transformed to a ``standard'' term $[x_{i_1}, \ldots, x_{i_k}, [x_{i_{k+1}}, \ldots, x_{i_n}]]$ for some $k\in\{1,\ldots,n-1\}$, where the sets $\{i_1, i_{k+2}, \ldots, i_n\}$ and $\{i_2, \ldots, i_{k+1}\}$ respectively contain the indices of the leftmost two variables of $t$ and either set of indices can be permuted arbitrarily.
We first write $t = [t_0, t_1, \ldots, t_m]$ with $|t_0|=1$.
We may assume that $|t_1|\ge |t_2|\ge \cdots\ge |t_m|$, thanks to the identity (i) and Lemma~\ref{lem:t}.
We distinguish some cases below.

\vskip5pt\noindent\textsf{Case 1}: $m>1$ and $|t_1|=1$.
Then $t=[x_{i_1}, \ldots, x_{i_n}]$, which is in standard form with $k = n - 1$.
The leftmost two variables of $t$ are indexed by $i_1\in \{i_1\}$ and $i_2\in \{i_2, \ldots, i_{k+1}\}$, and we can permute $i_2, \ldots, i_{k+1}$ by (i).

\vskip5pt\noindent\textsf{Case 2}: $m>1$ and $|t_1|>1$.
We can first apply (iii) repeatedly to transform $t$ to $x_{i_n} t'$, where $i_n$ is the index of the leftmost variable of $t$ and the leftmost variable of $t'$ is the second leftmost variable of $t$.
By the induction hypothesis, we may assume that $t' = [x_{i_{k+1}}, \ldots, x_{i_{n-1}}, [x_{i_1}, \ldots, x_{i_k}]]$, where $i_{k+1}, i_2, \ldots, i_k$  can be permuted in all possible ways and so can be $i_{k+2}, \ldots, i_{n-1}, i_1$, and the leftmost variable of $t'$ is indexed by one of $i_{k+1}, i_2, \ldots, i_k$.
We then apply (ii) to switch $x_{i_n}$ with $[x_{i_1}, \ldots, x_{i_k}]$ and get $[x_{i_1}, \ldots, x_{i_k}, [x_{i_{k+1}}, \ldots, x_{i_n}]]$.
By (i), we can switch $i_n$ and each of $i_{k+2},\ldots, i_{n-1}$. 
Thus we are done for this case.

\vskip5pt\noindent\textsf{Case 3}: $m=1$.
Similarly to the above case, we can apply the induction hypothesis to $t_1$ and then use (i) and (ii) to finish the argument for this case.
\vskip5pt

It follows that $\spac{*}$ is bounded above by the number of nonempty proper subsets of $\{1,\ldots,n\}$, which is clearly $2^n-2$.
Restricting the above argument to $t\in\B_n$, we must have $1\in\{i_1, i_{k+2},\ldots, i_n\}$ and $2\in\{i_2,\ldots, i_{k+1}\}$.
Thus $\spa{*}\le 2^{n-2}$; see also earlier work~\cite{CatMod}.
It is easy to see that $\spac{*}=2^n-2$ implies $\spa{*}=2^{n-2}$.

The usual subtraction $-$ on $\RR$ or $\CC$ satisfies the identities (i), (ii), and (iii).
We have $\spa{-} = 2^{n-2}$ and $\spac{-} =  2^n-2$ by previous work~\cite[Example~7.1.4]{AC-Spectrum}.
The same argument there is also valid for subtraction on any commutative group $(G,{+})$ of exponent greater than $2$ and in particular, for $\SC2302$.
\end{proof}

\begin{remark}
If $(G,{+})$ is a commutative group of exponent at most $2$, then the subtraction coincides with addition and $\spac{{-}} = 1$ for all $n \in \IN_{+}$.
\end{remark}

We provide another variety of groupoids $(G,*)$ with the same associative spectrum upper bound $2^{n-2}$ as Proposition~\ref{prop:subtration} but a different ac-spectrum upper bound $1,2,9,28,75,186,\ldots$~\cite[A058877]{OEIS}.
We show that both upper bounds are reached by two $3$-element groupoids $\SC271$ and $\SC356$, which are anti-isomorphic to $\SC1610$ (by $0\mapsto1$, $1\mapsto2$, $2\mapsto0$) and $\SC2032$ (by $0\mapsto2$, $1\mapsto0$, and $2\mapsto1$), respectively.

\begin{center}
\begin{tabular}{ccccccc}
\begin{tabular}{c|ccc}
$*$ & $0$ & $1$ & $2$ \\
\hline
$0$ & $0$ & $0$ & $0$ \\
$1$ & $1$ & $1$ & $0$ \\
$2$ & $2$ & $2$ & $2$ 
\end{tabular}
&
\begin{tabular}{c|ccc}
$*$ & $0$ & $1$ & $2$ \\
\hline
$0$ & $0$ & $0$ & $0$ \\
$1$ & $2$ & $1$ & $1$ \\
$2$ & $1$ & $2$ & $2$ 
\end{tabular}
&
\begin{tabular}{c|ccc}
$*$ & $0$ & $1$ & $2$ \\
\hline
$0$ & $0$ & $1$ & $1$ \\
$1$ & $0$ & $1$ & $2$ \\
$2$ & $0$ & $1$ & $2$ 
\end{tabular}
&
\begin{tabular}{c|ccc}
$*$ & $0$ & $1$ & $2$ \\
\hline
$0$ & $0$ & $1$ & $2$ \\
$1$ & $0$ & $1$ & $2$ \\
$2$ & $1$ & $0$ & $2$ 
\end{tabular}
\\
$\SC271$ & $\SC356$ & $\SC1610$ & $\SC2032$
\end{tabular}
\end{center}

\begin{theorem}\label{thm:A058877}
A groupoid $(G,*)$ satisfying the identities below must have $\spa{*}\le 2^{n-2}$ and $\spac{*}\le n(2^{n-1}-1)$ for $n=2,3,\ldots$, where the first upper bound is reached whenever the second one is.
\[ \mathrm{(i)}\ (xy)z\approx (xz)y \quad
\mathrm{(ii)}\ w(x(yz)) \approx w((xy)z) \quad
\mathrm{(iii)}\ (wx)(yz) \approx (wy)(xz) \]
Moreover, both upper bounds are reached for the $3$-element groupoids $\SC271$ and $\SC356$ \textup{(}hence the anti-isomorphic $\SC1610$ and $\SC2032$\textup{)}.
\end{theorem}

\begin{proof}
We transform an arbitrary term $t\in\F_n$, whose leftmost decomposition is $t = [x_a, t_1, \ldots, t_m]$, to a ``standard term'' using (i), (ii), (iii), and Lemma~\ref{lem:t}.
We may assume that $|t_1|\le |t_2|\le \cdots\le |t_m|$, thanks to the identity (i).
Let $x_{a_i}$ be the leftmost variable of $t_i$ for $i=1,\ldots,m$.
If there exists a positive integer $j<m$ such that $|t_j|>1$, we can use the identity (ii) to transform $t_{j+1}$ to $x_{a_{j+1}}t'_{j+1}$ and then use the identity (iii) to switch $t_j$ and $x_{a_{j+1}}$.
Repeating this, we obtain $[x_a, x_{b_1}, \ldots, x_{b_{m-1}}, (x_{b_m} t'_m)]$ from $t$, where $\{b_1, \ldots, b_m\} = \{a_1, \ldots, a_m\}$.
We can assume that $b_1<b_2<\cdots<b_m$, thanks to the identities (i) and (iii).
Applying the identity (ii) repeatedly to $x_{b_m} t'_m$ gives $\langle x_{b_m}, x_{c_1}, x_{c_2}, \ldots, x_{c_{n-m-1}} \rangle$. 
We may assume that $c_1<c_2<\cdots<c_{n-m-1}$ by the identities $w(x(yz)) \approx w((xy)z) \approx w((xz)y) \approx w(x(zy))$.

It follows that every $t \in \F_n$ induces the same $n$-ary operation as a standard term 
\[ [x_a, x_{b_1}, \ldots, x_{b_{m-1}}] \langle x_{b_m}, x_{c_1}, \ldots, x_{c_{n-m-1}} \rangle, \]
where $b_1<\cdots<b_m$ and $c_1<\cdots<c_{n-m-1}$.
This implies that $\spac{*} \le n(2^{n-1}-1)$ since there are $n$ possibilities for $a$ and $2^{n-1}-1$ possibilities for $(b_1,\ldots,b_m)$.

Restricting the above argument to $t\in \B_n$, we must have $a=1$ and $b_1=2$ since $a_1=2\in\{b_1,\ldots,b_m\}$.
Thus $\spa{*} \le 2^{n-2}$, and it is easy to see that the equality must hold when $\spac{*}=n(2^{n-1}-1)$.

One can check that $\SC271$ and $\SC356$ both satisfy the identities (i), (ii), and (iii).
It remains to show that $h(t)\ne h(t')$ for some assignment $h:X_n\to\{0,1,2\}$, where $s$ and $t$ are terms in $\F_n$ corresponding to distinct standard terms
\[ [x_a, x_{b_1}, \ldots, x_{b_{m-1}}] \langle x_{b_m}, x_{c_1}, \ldots, x_{c_{n-m-1}} \rangle 
\ne [x_{a'}, x_{b'_1}, \ldots, x_{b'_{\ell-1}}] \langle x_{b'_\ell}, x_{c'_1}, \ldots, x_{c'_{n-\ell-1}} \rangle. \]
Assume $m\le \ell$, without loss of generality.

First suppose that $a\ne a'$.
Define $h(x_a):=0$ and $h(x)=2$ for all $x\ne x_a$.
For both $\SC271$ and $\SC356$, one can check that $h(s)=0\ne h(t)$.

Next, suppose $a=a'$. 
Then $ \{c_1,\ldots,c_{n-m-1}\}$ and $\{c'_1,\ldots,c'_{n-\ell-1}\}$  must be different sets. 
Suppose some $i$ belongs to the latter but not the former, without loss of generality.
We must have $i\in\{b_1,\ldots,b_m\}$.

For $\SC271$, we have $h(s)=0 \ne 1=h(t)$, where $h(x_i):=2$ and $h(x):=1$ for all $x\ne x_i$.

For $\SC356$, we have $h(s) =1 \ne 2=h(t)$, where $h(x_i):=0$ and $h(x):=2$ for all $x\ne x_i$.
\end{proof}

\section{Upper bounds related to set partitions}\label{sec:set-partition}

In this section, we present a few varieties of groupoids, whose ac-spectra are related to set partitions.
Recall that the \emph{restricted Bell number} $B_{n,m}$ counts partitions of the set $\{1,2,\ldots,n\}$ into unordered nonempty blocks of size at most $m$~\cite{Mezo}; it gives the well-known \emph{Bell number} $B_n$ when $m\ge n$.
In particular, we have $B_{n,2}=1$ for $n=0,1$ and $B_{n,2}=B_{n-1,2}+(n-1)B_{n-2,2}$ for $n\ge2$; see the sequence A000085 in OEIS~\cite{OEIS} for other interpretations and closed formulas for $B_{n,2}$.
We also need the following definition by Cs\'{a}k\'{a}ny and Waldhauser~\cite{AssociativeSpectra1}.

\begin{definition}
Define a term $t$ to be a \emph{nest} if either $|t|=1$ (a \emph{trivial nest}) or there exists a term $t$ together with a variable $x$ such that $t= xt'$ or $t=t'x$, $|t'|=|t|-1$, and $t'$ is a nest.
Each variable in $t$ must be contained in a unique maximal nest, which is simply called a \emph{nest} of $t$.
Every nontrivial nest must have a unique subterm of the form $x_i x_j$, and the variables $x_i$ and $x_j$ are called the \emph{eggs} of this nest.
\end{definition}

Our first result is concerned with a variety of groupoids including the following two $3$-element groupoids.
\begin{center}
\begin{tabular}{cc}
\begin{tabular}{c|ccc}
* & 0 & 1 & 2 \\
\hline
0 & 0 & 0 & 0 \\
1 & 0 & 1 & 0 \\
2 & 0 & 0 & 1 
\end{tabular}
&
\begin{tabular}{c|ccc}
$*$ & $0$ & $1$ & $2$ \\
\hline
$0$ & $0$ & $1$ & $1$ \\
$1$ & $1$ & $0$ & $0$ \\
$2$ & $1$ & $0$ & $1$ 
\end{tabular} \\
$\SC79$ & $\SC1701$\\
\end{tabular}
\end{center}

\begin{proposition}\label{prop:79}
A groupoid $(G,*)$ satisfying the identities below must have $\spa{*}\le F_{n+1}-1$ and $\spac{*} \le B_{n,2}-1$ for $n=2,3,\ldots$, where the first inequality holds as an equality whenever the second does.
\[ \mathrm{(i)}\ xy\approx yx, \quad
\mathrm{(ii)}\ ((wx)y)z \approx ((wx)z)y \]
Moreover, both upper bounds are reached by $\SC79$ and $\SC1701$.
\end{proposition}

\begin{proof}
Suppose $s,t\in \F_n$ have the same eggs of nests.
We show by induction on $n$ that $s^*=t^*$.
Let $x_i$ and $x_j$ be the eggs of a nest of $s$; they must be the eggs of a nest of $t$.
The case $n=2$ is trivial; assume $n\ge3$ below.
Thanks to the identity (i), we may assume $s=[x_i, x_j, s_1, \ldots, s_\ell]$ and $t=[x_i, x_j, t_1, \ldots, t_m]$.
We may also assume that $|s_1|\ge \cdots \ge |s_\ell|$ and $|t_1|\ge \cdots \ge |t_m|$ by (ii).
Assume $|s_1|\le |t_1|$, without loss of generality.

\vskip5pt\noindent\textsf{Case 1:} $|t_1|\ge |s_1|>1$.
Replacing $x_i x_j$ with a new variable $x_0$ in both $s$ and $t$ gives full linear terms $s'$ and $t'$ in $n-1$ variables that share the same eggs of nests.
It follows from the induction hypothesis that $(s')^*=(t')^*$, and this implies $s^*=t^*$.

\vskip5pt\noindent\textsf{Case 2:} $|s_1|=1$.
Then $|s_2|=\cdots=|s_\ell|=1$ and $s$ has only two eggs $x_i$ and $x_j$.
We must have $|t_1|=1$ (otherwise $t_1$ contains eggs different from $x_i$ and $x_j$) and thus $|t_2|=\cdots=|t_m|=1$.
We can use (ii) to make sure $s_1=t_1=x_k$ for some $k\notin\{i,j\}$.
Replacing $x_i x_j$ with a new variable $x_0$ in both $s$ and $t$ gives $s'$ and $t'$ with eggs $x_0$ and $x_k$.
By the induction hypothesis, we have $(s')^*=(t')^*$.
This implies that $s^*=t^*$.
\vskip5pt

Therefore, $\spac{*}$ is bounded above by $B_{n,2}-1$, which is the number of partitions of $\{1,\ldots,n\}$ into blocks of size one or two with at least one block of size two (since there is at least one nest with two eggs).

Restricting the above argument to bracketings of $x_1, \ldots, x_n$, we have $\spa{*}\le F_{n+1}-1$ since the partitions associated with bracketings of $x_1, \ldots, x_n$ must have two consecutive integers in each block of size two; see also Cs\'{a}k\'{a}ny and Waldhauser~\cite[\S5.6]{AssociativeSpectra1}.
It is easy to see that $\spac{*}=B_{n,2}-1$ implies $\spa{*}= F_{n+1}-1$.

It is routine to verify that groupoids $\SC79$ and $\SC1701$ satisfy identities (i) and (ii).
It remains to verify that if $s, t \in \F_n$ are terms whose egges of nests are not the same, then $s$ and $t$ induce distinct operations on $\SC79$ and on $\SC1701$.
Suppose that $x_i$ and $x_j$ are eggs of a nest in $s$ but not eggs of any nest in $t$.
For $\SC79$,  Cs\'{a}k\'{a}ny and Waldhauser~\cite{AssociativeSpectra1} observed that $h(s) =1  \ne 0 = h(t)$, where $h(x_i)=h(x_j):=2$ and $h(x):=1$ for all $x\notin\{x_i,x_j\}$.
For $\SC1701$, we have $h(s) = 1\ne 0= h(t)$, where $h(x_i)=h(x_j):=2$ and $h(x):=0$ for all $x\notin\{x_i,x_j\}$.
Thus $\spac{*} = B_{n,2}-1$ and $\spa{*}=F_{n+1}-1$ for $\SC79$ and $\SC1701$.
\end{proof}

A set partition is \emph{rooted} if it has a distinguished singleton block called the \emph{root}.
The number of rooted partitions of $\{1,2,\ldots,n\}$ is $nB_{n-1}=1,2,6,20,75,312,\ldots$~\cite[A052889]{OEIS}.
We show below that this number is the upper bound for the ac-spectra of a variety of groupoids and can be attained by the $3$-element groupoids $\SC41$ and $\SC96$.
The Cayley tables of these two groupoids together with the anti-isomorphic groupoids $\SC398$ and $\SC1069$ are given below.
\begin{center}
\begin{tabular}{ccccccc}
\begin{tabular}{c|ccc}
$*$ & $0$ & $1$ & $2$ \\
\hline
$0$ & $0$ & $0$ & $0$ \\
$1$ & $0$ & $0$ & $1$ \\
$2$ & $1$ & $1$ & $2$ 
\end{tabular}
&
\begin{tabular}{c|ccc}
$*$ & $0$ & $1$ & $2$ \\
\hline
$0$ & $0$ & $0$ & $0$ \\
$1$ & $0$ & $1$ & $0$ \\
$2$ & $2$ & $0$ & $2$ 
\end{tabular}
&
\begin{tabular}{c|ccc}
$*$ & $0$ & $1$ & $2$ \\
\hline
$0$ & $0$ & $0$ & $1$ \\
$1$ & $0$ & $0$ & $1$ \\
$2$ & $0$ & $1$ & $2$ 
\end{tabular}
&
\begin{tabular}{c|ccc}
$*$ & $0$ & $1$ & $2$ \\
\hline
$0$ & $0$ & $0$ & $2$ \\
$1$ & $0$ & $1$ & $0$ \\
$2$ & $0$ & $0$ & $2$ 
\end{tabular}
\\
$\SC41$ & $\SC96$ & $\SC398$ & $\SC1069$
\end{tabular}
\end{center}

\begin{theorem}\label{thm:set-partition}
A groupoid $(G,*)$ satisfying the identities below must have $\spa{*}\le 2^{n-2}$ for $n=2,3,\ldots$ and $\spac{*}\le nB_{n-1}$ for $n=1,2,\ldots$, where the second inequality holds as an equality whenever the second does.
\[ \mathrm{(i)}\ x(yz) \approx x(zy), \quad
\mathrm{(ii)}\ (xy)z \approx (xz)y. \]
Moreover, both upper bounds are reached by $\SC41$ and $\SC96$ \textup{(}hence the anti-isomorphic $\SC398$ and $\SC1069$\textup{)}.
\end{theorem}

\begin{proof}
Let $t$ be an arbitrary term in $\F_n$ with leftmost decomposition $t = [x_a, t_1, \dots, t_m]$.
Define a rooted set partition of $\{1,2,\ldots,n\}$ associated with $t$: we have a block consisting of the indices of the variables in $t_j$ for all $j=1,2,\ldots,m$ together with a singleton block $\{a\}$ that is the root of this partition. 
By (i), (ii), and Lemma~\ref{lem:t}, $t$ induces on $(G,*)$ the same term operation as $[x_a, t_{\sigma(1)}^{\mathrm{L}\mathord{<}}, \dots, t_{\sigma(m)}^{\mathrm{L}\mathord{<}}]$ for any permutation $\sigma \in \SS_m$.
It follows that terms in $\F_n$ associated with the same rooted partition must induce the same $n$-ary operation on $(G,*)$.
Thus $\spac{*}\le nB_{n-1}$.

The rooted set partition associated with a bracketing of $x_1*\cdots*x_n$ must have $\{1\}$ as its root and the other blocks are intervals.
The number of such ``interval partitions'' can be found by counting the number of ways of inserting bars into the $n-2$ spaces between $2,\ldots,n$.
Thus $\spa{*}\le 2^{n-2}$.

If $\spac{*}=nB_{n-1}$ for $n\ge1$, then $s^*\ne t^*$ whenever $s,t\in \F_n$ are associated with distinct rooted set partitions, and restricting this to bracketings of $x_1*\cdots*x_n$ gives $\spa{*}=2^{n-2}$.

It is routine to check that $\SC41$ and $\SC96$ both satisfy the identities (i) and (ii).
It remains to show that $s^*\ne t^*$ whenever $s$ and $t$ are terms in $\F_n$ associated with distinct rooted set partitions.
Suppose $s = [ x_a, s_1, \ldots, s_\ell ]$ and $t = [ x_b, t_1, \ldots, t_m ]$, where $s_1, \ldots, s_\ell$ and $t_1, \ldots, t_m$ are ordered according to the smallest index of the variables they contain.
If $a\ne b$ then $s^*\ne t^*$ since
\begin{itemize}
\item
$h(s) = 0 \ne 1 = h(t)$ if $(\{0,1,2\},*)=\SC41$, $h(x_a)=0$ and $h(x_i) = 2$ for all $i \ne a$, and
\item
$h(s) = 0 \ne 2 = h(t)$ if $(\{0,1,2\},*)=\SC96$, $h(x_a)=0$ and $h(x_i) = 2$ for all $i \ne a$.
\end{itemize}
Assume $a=b$ below.
Let $j$ be the smallest integer such that $s_j$ and $t_j$ do not contain the same set of variables.
The least index $c$ of the variables of $s_j$ must agree with that of $t_j$.
There exists another variable $x_d$ in exactly one of $s_j$ and $t_j$, say the former.
Then $x_d$ is in $t_k$ for some $k>j$. 
We have
\begin{itemize}
\item
$h(s) = 1 \ne 0 = h(t)$ if $(\{0,1,2\},*)=\SC41$, $h(x_c)=h(x_d)=0$, and $h(x_i) = 2$ for all $i\notin\{c,d\}$, and
\item
$h(s) = 2 \ne 0 = h(t)$ if $(\{0,1,2\},*)=\SC96$, $h(x_a)=h(x_c)=2$, and $h(x_i) = 1$ for all $i\notin\{a,c\}$.
\end{itemize}
Thus $s^*\ne t^*$.
\end{proof}

Next, we provide an ordered version of Theorem~\ref{thm:set-partition} that has the same associative spectrum upper bound but a different ac-spectrum upper bound.
Recall that the \emph{ordered Bell number} or \emph{Fubini number} $B'_n$ counts ordered partitions of the set $\{1,2,\ldots,n\}$~\cite[A000670]{OEIS}.
The number of rooted ordered set partitions of $\{1,\ldots,n\}$ is $nB'_{n-1}=1, 2, 9, 52, 375,\ldots$~\cite[A052882]{OEIS}.
We show that $nB'_{n-1}$ is also the upper bound for the ac-spectra of a variety of groupoids and can be reached by the $3$-element groupoids $\SC262$, $\SC1812$, and $\SC2446$, which are anti-isomorphic to $\SC1441$ (by $0\mapsto2$, $1\mapsto0$, and $2\mapsto1$), $\SC1793$ and $\SC2430$, respectively.
\begin{center}
\begin{tabular}{ccccccc}
\begin{tabular}{c|ccc}
$*$ & $0$ & $1$ & $2$ \\
\hline
$0$ & $0$ & $0$ & $0$ \\
$1$ & $1$ & $1$ & $0$ \\
$2$ & $1$ & $1$ & $2$ 
\end{tabular}
&
\begin{tabular}{c|ccc}
$*$ & $0$ & $1$ & $2$ \\
\hline
$0$ & $0$ & $0$ & $2$ \\
$1$ & $2$ & $1$ & $2$ \\
$2$ & $0$ & $0$ & $2$ 
\end{tabular}
&
\begin{tabular}{c|ccc}
$*$ & $0$ & $1$ & $2$ \\
\hline
$0$ & $0$ & $1$ & $1$ \\
$1$ & $1$ & $2$ & $1$ \\
$2$ & $1$ & $2$ & $1$ 
\end{tabular}
&
\begin{tabular}{c|ccc}
$*$ & $0$ & $1$ & $2$ \\
\hline
$0$ & $0$ & $1$ & $1$ \\
$1$ & $1$ & $2$ & $2$ \\
$2$ & $1$ & $1$ & $1$ 
\end{tabular}
&
\begin{tabular}{c|ccc}
$*$ & $0$ & $1$ & $2$ \\
\hline
$0$ & $1$ & $0$ & $0$ \\
$1$ & $0$ & $2$ & $1$ \\
$2$ & $0$ & $2$ & $1$ 
\end{tabular}
& 
\begin{tabular}{c|ccc}
$*$ & $0$ & $1$ & $2$ \\
\hline
$0$ & $1$ & $0$ & $0$ \\
$1$ & $0$ & $2$ & $2$ \\
$2$ & $0$ & $1$ & $1$ 
\end{tabular}
\\
$\SC262$ & $\SC1441$ & $\SC1793$ & $\SC1812$ & $\SC2430$ & $\SC2446$
\end{tabular}
\end{center}

\begin{theorem}\label{thm:ordered-set-partition}
A groupoid $(G,*)$ satisfying the identities below must have $\spa{*}\le 2^{n-2}$ for $n=2,3,\ldots$ and $\spac{*}\le n B'_{n-1}$ for $n=1,2,\ldots$, where the first inequality holds as an equality whenever the second does.
\[ \mathrm{(i)}\ x(yz) \approx x(zy), \quad 
\mathrm{(ii)}\ w(x(yz)) \approx w((xy)z) \] 
Moreover, both equalities hold for $\SC262$, $\SC1812$, and $\SC2446$ \textup{(}hence the anti-isomorphic $\SC1441$, $\SC1793$, and $\SC2430$\textup{)}.
\end{theorem}

\begin{proof}
By Lemma~\ref{lem:t}, we can transform an arbitrary term $t\in \F_n$ with leftmost decomposition $t = [ t_0, t_1, \ldots, t_m]$, where $|t_0|=1$, to $[t_0, t_1^{\mathrm{L}\mathord{<}}, \dots, t_m^{\mathrm{L}\mathord{<}}]$.
Thus terms in $\F_n$ induce the same $n$-ary operation if they are associated with the same rooted ordered set partitions.
It follows that $\spac{*}\le nB'_{n-1}$.
Restricting the above argument to $\B_n$ gives $\spa{*}\le 2^{n-2}$, where the equality holds if $\spac{*}=nB'_{n-1}$.

It is routine to check that $\SC262$, $\SC1812$, and $\SC2446$ all satisfy the identities (i) and (ii).
It remains to show that $s^* \ne t^*$ whenever $s,t\in \F_n$ are associated with distinct rooted ordered set partitions of $\{1,2,\ldots,n\}$. 
We can write $s = [x_a, s_1, \ldots, s_\ell]$ and $t = [ x_b, t_1, \ldots, t_m]$.
If $a\ne b$ then $s^*\ne t^*$ by the following:
\begin{itemize}
\item
For $\SC262$, we have $h(s) = 1 \ne 0 = h(t)$, where $h(x_a):=1$ and $h(x) := 0$ for all $x\ne x_a$.
\item
For $\SC1812$ and $\SC2446$, one of $h(s)$ and $h(t)$ is $1$ and the other is $2$, where $h(x):=1$ for all $x$ if $\ell$ and $m$ have different parities or $h(x_a):=1$ and $h(x) := 2$ for all $x \ne x_a$ otherwise.
\end{itemize}

Assume $a=b$ below. 
Let $j$ be the smallest integer such that $\var(s_j)\ne \var(t_j)$.

For $\SC262$, we distinguish two cases.

\vskip5pt\noindent\textsf{Case 1}: $\var(t_i)\not\subseteq \var(s_j)$ for all $i$.
Define $h(x):=2$ for all $x\in\{x_a\}\cup\var(s_j)$ and $h(x):=0$ for all $x\notin \var(s_j)$.
Then $h(x_a)=2$, $h(s_j)=2$, $h(s_i)=0$ for all $i\ne j$, and $h(t_i)\in\{0,1\}$ for all $i$.
One can check that $h(s)=1\ne 0=h(t)$ when $j=1$ and $h(s)=0\ne 1=h(t)$ when $j>1$.

\vskip5pt\noindent\textsf{Case 2}: $\var(t_k)\subseteq \var(s_j)$ for some $k$.
If $\var(t_k)\subsetneq \var(s_j)$, then we are back to Case 1 by switching $s$ and $t$ and using $t_k$ instead of $s_j$, since $\var(s_i)\not\subseteq \var(t_k)$ for all $i$.
Thus we may assume that $\var(s_j)=\var(t_k)$, which implies $j<k$ since $\var(s_i) = \var(t_i)$ for all $i<j$.
Define
\[ h(x):=\begin{cases}
2, & \text{if } x\in\{x_a\}\cup\var(s_1)\cup\cdots\cup\var(s_j)=\var(t_1)\cup\cdots \cup \var(t_{j-1})\cup\var(t_k); \\
0, & \text{if } x\notin\{x_a\}\cup\var(s_1)\cup\cdots\cup\var(s_j)=\var(t_1)\cup\cdots \cup \var(t_{j-1})\cup\var(t_k).
\end{cases} \]
We have $h(s_1)=\cdots=h(s_j)=2$, $h(s_i)=0$ for all $i=j+1,\ldots,\ell$, and thus $h(s)=1$. 
On the other hand, we have $h(t_1)=\cdots=h(t_{j-1})=h(t_k)=2$, $h(t_i)=0$ for all $i\in\{j,\ldots,m\}\setminus\{k\}$, and thus $h(t)=0\ne h(s)$.
\vskip5pt

For $\SC1812$, we may assume that $\ell$ and $m$ have the same parity by the all-$1$ substitution as discussed earlier.
We distinguish some cases below.

\vskip5pt\noindent\textsf{Case 1}: $\var(t_i)\not\subseteq \var(s_j)$ for all $i$.
We further distinguish two subcases below.
\begin{itemize}
\item
Suppose that $j$ is odd.
Define $h(x):=0$ for all $x\in\var(s_j)$ and $h(x):=1$ for all $x\notin \var(s_j)$.
Then $h(x_a)=1$, $h(s_j)=0$, $h(s_i)\in\{1,2\}$ for all $i\ne j$, and $h(t_i)\in\{1,2\}$ for all $i$.
One can check that $h(s)=1$ if $\ell$ is odd or $h(s)=2$ if $\ell$ is even.
On the other hand, we have $h(t)=1$ if $m$ is even or $h(t)=2$ otherwise.
Since $\ell$ and $m$ have the same parity, it follows that $h(s)\ne h(t)$.
\item
Suppose that $j$ is even. 
Defined by $h(x):=0$ for all $x\in\var(s_j)$ and $h(x):=2$ for all $x\notin \var(s_j)$.
Then $h(x_a)=1$, $h(s_j)=0$, $h(s_i)\in\{1,2\}$ for all $i\ne j$, and $h(t_i)\in\{1,2\}$ for all $i$.
One can check that $h(s)=1$ if $\ell$ is even or $h(s)=2$ if $\ell$ is odd.
On the other hand, we have $h(t)=1$ if $m$ is odd or $h(t)=2$ if $m$ is even.
Since $\ell$ and $m$ have the same parity, we must have $h(s) \ne h(t)$.
\end{itemize}

\noindent\textsf{Case 2}: $\var(t_k)\subseteq \var(s_j)$ for some $k$.
If $\var(t_k)\subsetneq \var(s_j)$, then we are back to Case 1 by switching $s$ and $t$ and using $t_k$ instead of $s_j$, since $\var(s_i)\not\subseteq \var(t_k)$ for all $i$.
Thus we may assume that $\var(s_j)=\var(t_k)$, which implies $j<k$.
We further distinguish two subcases below.
\begin{itemize}
\item
Suppose that $j$ and $k$ have different parities.
Define $h(x):=0$ for all $x\in\var(s_j)$ and $h(x):=2$ for all $x\notin\var(s_j)$.
Then $h(x_a)=2$, $h(s_j)=h(t_k)=0$, $h(s_i)\in\{1,2\}$ for all $i\ne j$, and $h(t_i)\in\{1,2\}$ for all $i\ne k$.
One can check that $h(s)=1$ if $j$ has the same parity as $\ell$ or $h(s)=2$ otherwise.
Similarly, $h(t)=1$ if $k$ has the same parity as $m$ or $h(t)=2$ otherwise.
Since $\ell$ and $m$ have the same parity, we must have $h(s) \ne h(t)$.
\item
Suppose that $j$ and $k$ have the same parity.
Define
\[ h(x):=\begin{cases}
0, & \text{if } x\in\{x_a\}\cup\var(s_1)\cup\cdots\cup\var(s_j)=\var(t_1)\cup\cdots \cup \var(t_{j-1})\cup\var(t_k); \\
1, & \text{if } x\notin\{x_a\}\cup\var(s_1)\cup\cdots\cup\var(s_j)=\var(t_1)\cup\cdots \cup \var(t_{j-1})\cup\var(t_k).
\end{cases} \]
Then $h(x_a)=h(s_j)=h(t_k)=0$, $h(s_i)=h(t_i)=0$ for all $i=1,\ldots,j-1$, $h(s_i)\in\{1,2\}$ for all $i=j+1,\ldots,\ell$, and $h(t_i)\in\{1,2\}$ for all $i\in\{j,\ldots,m\}\setminus\{k\}$.
One can check that $h(s)=1$ if $j$ and $\ell$ have different parities or $h(s)=2$ otherwise (note that $j<\ell$).
Similarly, $h(t)=1$ if $k$ and $m$ have the same parity or $h(t)=2$ otherwise.
Since $\ell$ and $m$ have the same parity, we must have $h(s)\ne h(t)$.
\end{itemize}

For $\SC2446$, we may again assume that $\ell$ and $m$ have the same parity by the all-$1$ substitution.
There exists a variable $x_c$ in exactly one of $s_j$ and $t_j$, say the former.
Then $x_c$ is in $t_k$ for some $k>j$.
We distinguish two cases below.

\vskip5pt\noindent\textsf{Case 1}: $j$ and $k$ have different parities.
Define $h(x_a)=h(x_c):=0$ and $h(x):=1$ for all $x\notin\{x_a,x_c\}$.
We have $h(s_j)=0$ and $h(s_i)\in\{1,2\}$ for all $i\ne j$.
Thus $h(s)=1$ if $j$ has the same parity as $\ell$ or $h(s)=2$ otherwise.
Similarly, we have $h(t_k)=0$ and $h(t_i)\in\{1,2\}$ for all $i\ne k$.
Thus $h(t)=1$ if $k$ has the same parity as $m$ or $h(t)=2$ otherwise.
Then $h(s)\ne h(t)$ since $\ell$ and $m$ have the same parity.

\vskip5pt\noindent\textsf{Case 2}: $j$ and $k$ have the same parity.
Pick any variable $x_d$ in $t_{k-1}$, which must be in $s_{j'}$ for some $j'\ge j$.
The argument in the above paragraph is valid for $j'$ and $k-1$ if they have different parities.
Otherwise $j'$ and $k$ must have different parities, and it follows that $j'>j$.
Define $h(x_c)=h(x_d):=0$ and $h(x):=1$ for all $x\notin\{x_c, x_d\}$.
We have $h(s_j)=h(s_{j'})=0$ and $h(s_i)\in\{1,2\}$ for all $i\notin\{j,j'\}$.
Thus $h(s) = 1$ if $j'$ has the same parity as $\ell$, or $h(s)=2$ otherwise.
Similarly, we have $h(t_{k-1})=h(t_k)=0$ and $h(s_i)\in\{1,2\}$ for all $i\notin\{k-1,k\}$.
Thus $h(t) = 1$ if $k$ has the same parity as $m$, or $h(t)=2$ otherwise.
Then $h(s)\ne h(t)$ since $\ell$ and $m$ have the same parity but $j'$ and $k$ have different parities.
\end{proof}

\section{Congruence on depths}\label{sec:depth}
In this section we discuss the natural occurrence of leaf depths in the study of associative and ac-spectra of groupoids and how it can help us generalize some of our results.

Using both identities and the left/right depth, Hein and the first author~\cite{CatMod} determined the associative spectrum of a generalization of addition and subtraction 
to be the \emph{modular Catalan number}
\[ C_{k,n} := \sum_{0\le j\le (n-1)/k} \frac{(-1)^j}{n} {n\choose j} {2n-jk\choose n+1}, \] 
and we determined its ac-spectrum in our previous work~\cite{AC-Spectrum}.
These results are rephrased below to include Proposition~\ref{prop:subtration} as a special case (using right depth instead of identities).

\begin{theorem}[\cite{CatMod, AC-Spectrum}]\label{thm:k-right-depth}
Let $(G,*)$ be a groupoid such that for all $s,t\in F_n$, we have $s^*=t^*$ whenever $\rho_i(s) \equiv \rho_i(t) \pmod k$ for $i=1,\ldots,n$.
Then $\spa{*} \le C_{k,n-1}$ and 
\[ \spac{*} \le k! S(n,k) + n \sum_{0 \leq i \leq k-2} i! S(n-1,i) \]
for $n=1,2,\ldots$, where the first equality holds as an equality if the second one does. 
Moreover, both upper bounds are reached if ``whenever'' can be replaced with ``if and only if'' in the above condition.
In particular, both upper bounds are attained by $(\CC,*)$, where $a*b := a+e^{2\pi i/k} b$ for all $a,b\in\CC$.
\end{theorem}

Now we use the left depth to generalize Proposition~\ref{prop:189} and Proposition~\ref{prop:3242} as follows.

\begin{theorem}
Let $(G,*)$ be a groupoid such that for all $s,t\in F_n$, we have $s^*=t^*$ whenever $s$ and $t$ have the same leftmost variable $x_i$, whose left depths in $s$ and $t$ are congruent modulo $k$.
Then $\spa{*} \le k$ and $\spac{*} \le kn$ for $n=k+1,\ldots$, where the first inequality holds as an equality if the second does.
Moreover, both upper bounds are reached if ``whenever'' can be replaced with ``if and only if'' in the above condition.  
\end{theorem}

\begin{proof}
First, suppose that $s^*=t^*$ whenever $s$ and $t$ have the same leftmost variable $x_i$ and the left depths of $x_i$ in $s$ and $t$ are congruent modulo $k$.
Then every term in $\F_n$ induces the same $n$-ary operation on $(G,*)$ as a standard term $[x_i, x_{i_1}, \ldots, x_{i_m}, \langle x_{i_{m+1}},\ldots, x_{i_{n-1}} \rangle]$, where $i_1<\cdots<i_{n-1}$ and $m\in \{0,\ldots,k-1\}$.
The above standard term is determined by $x_i$ and $m$, for which there are $n$ and $k$ possibilities, respectively (the latter requires $n\ge k+1$).
Thus $\spac{*}\le kn$.
Similarly, the standard term of each bracketing in $\B_n$ must begin with $x_1$. Thus $\spa{*}\le k$.
It is easy to see that $\spac{*}=kn$ implies $\spa{*}=k$.

Now suppose that $s^*=t^*$  if and only if $s$ and $t$ have same leftmost variable $x_i$ and the left depths of $x_i$ in $s$ and $t$ are congruent modulo $k$.
The ``only if'' part implies that $s^*\ne t^*$ if $s$ and $t$ correspond to different standard terms.
Thus $\spa{*} = k$ and $\spac{*} = kn$.
\end{proof}

\begin{remark}
Hein and the first author~\cite{CatMod} observed that the congruence relation modulo $k$ on the left depths of the bracketings in $\B_n$ is characterized by the identity $s_0 [s_1,\ldots,s_{k+1}] \approx [s_0, s_1, \ldots, s_{k+1}]$ and showed that $C_{k,n-1}$ is the number of terms in $\B_n$ avoiding subterms of the form $s_0 [s_1,\ldots,s_{k+1}]$. 
We also have $\spa{*} \le C_{k,n-1}$ for a groupoid $(G,*)$ satisfying a different identity 
\begin{equation}\label{eq:comb+1+2}
s_0 [s_1,\ldots,s_{k+1}] \approx s_0 (s_1 [s_2,\ldots,s_{k+1}]) 
\end{equation}
since we can still use this identity to transform every bracketing in $\B_n$ to some bracketing in $\B_n$ that avoids subterms of the form $s_0 [s_1,\ldots,s_{k+1}]$.
Although not needed for the proof of the upper bound $\spa{*} \le C_{k,n-1}$, we can even show that distinct bracketings $t, t'\in\B_n$ both avoiding $s_0 [s_1,\ldots,s_{k+1}]$ cannot be obtained from each other by the identity~\eqref{eq:comb+1+2}, using the technique due to Hein and the first author~\cite{CatMod}.
In fact, we know that $t$ and $t'$ correspond to two binary trees with $n$ leaves labeled $1,\ldots,n$ from left to right, which in turn correspond to two \emph{rooted plane trees} $T$ and $T'$ with $n$ vertices labeled $1,\ldots,n$ in the \emph{preorder} by contracting each northeast southwest ``long edge'' in the drawings of $t$ and $t'$.
If $t$ can be obtained from $t'$ by the identity~\eqref{eq:comb+1+2}, then a non-root vertex in $T$ must have its \emph{degree} (the number of children) less than $k$ and congruent  to the degree of the vertex with the same label in $T'$ modulo $k-1$, and the leaves (degree-zero vertices) in $T$ must correspond to the leaves in $T'$.
Thus the degrees of the vertices of $T$ must agree with those of $T'$, and this forces $T=T'$.

For $k=3$, we suspect that $\spa{*}= C_{k,n-1}$ holds for $\SC64$, which is anti-isomorphic to $\SC399$.
\begin{center}
\begin{tabular}{ccccccc}
\begin{tabular}{c|ccc}
$*$ & $0$ & $1$ & $2$ \\
\hline
$0$ & $0$ & $0$ & $0$ \\
$1$ & $0$ & $0$ & $2$ \\
$2$ & $1$ & $1$ & $0$ 
\end{tabular}
&
\begin{tabular}{c|ccc}
$*$ & $0$ & $1$ & $2$ \\
\hline
$0$ & $0$ & $0$ & $1$ \\
$1$ & $0$ & $0$ & $1$ \\
$2$ & $0$ & $2$ & $0$ 
\end{tabular}
\\
$\SC64$ & $\SC399$
\end{tabular}
\end{center}
In fact, our computations show that the initial terms of the associative spectrum and ac-spectrum of $\SC64$ are $1,1,2,5,13,35,96,267$ and $1,2,12,84,710$, respectively; the former sequence coincides with $C_{3,n-1}$ while the latter differs from the upper bound of $\spac{*}$ for $k=3$ in Theorem~\ref{thm:k-right-depth}, whose initial terms are $1, 2, 9, 40, 155, 546, 1813, 5804, 18159$.
One can check that $\SC64$ satisfies at least the four identities below. 
\[ w(x(yz)) \approx w(y(xz)), \quad
w((xy)z) \approx w((zy)x), \quad
((wx)y)z \approx ((wz)y)x, \quad
v(w((xy)z)) \approx v(((wx)y)z) \]
But these identities seem unrelated to the left/right depth modulo $k=3$.
\end{remark}

The first author, Mickey, and Xu~\cite{DoubleMinus} used the depth to find the associative spectrum of the \emph{double minus} operation $a*b:=-a-b$, and we determined the ac-spectrum of this operation in previous work~\cite{AC-Spectrum}. 
Both proofs are valid for any field with at least three elements, giving the following result.

\begin{theorem}[\cite{DoubleMinus}]
Suppose that two terms $s,t\in\F_n$ induce the same $n$-ary operation on a groupoid $(G,*)$ whenever $d_i(s)\equiv d_i(t) \pmod2$ for $i=1,\ldots,n$.
Then $\spa{*}\le \lfloor 2^n/3 \rfloor$ and $\spac{*} \le (2^n-(-1)^n)/3$ for $n=1,2,\ldots$, where the first equality holds as an equality if the second one does.
Moreover, both upper bounds are reached if ``whenever'' can be replaced with ``if and only if'' in the above condition.
In particular, both upper bounds are achieved by the double minus operation on any field with at least three elements.
\end{theorem}

The two upper bounds in the above theorem are both well studied~\cite[A000975, A001045]{OEIS} from many other perspectives; the latter is known as the \emph{Jacobsthal sequence}.
The double minus operation on a field of three elements is actually the $3$-element groupoid $\SC2346$.
\begin{center}
\begin{tabular}{ccc}
\begin{tabular}{c|ccc}
$*$ & $0$ & $1$ & $2$ \\
\hline
$0$ & $0$ & $2$ & $1$ \\
$1$ & $2$ & $1$ & $0$ \\
$2$ & $1$ & $0$ & $2$ 
\end{tabular} \\
$\SC2346$
\end{tabular}
\end{center}
To generalize the above theorem, one could use a primitive root of unity $\omega:=e^{2\pi i/k}$ to define an operation $a*b:=\omega a + \omega b$ on the field of complex numbers, which reduces to the double minus operation when $k=2$;
for $k\ge3$, the $n$-th term of the associative spectrum was shown in \cite{LehWal-QG} to coincide with the number of equivalence classes of the equivalence relation on $n$-leaf binary trees that relates two trees if the depths of corresponding leaves are congruent modulo $k$.
Closed formulas for the associative spectrum and the ac-spectrum of this operation are yet to be determined.

\section{Questions and remarks}\label{sec:question}

We have some more questions other than those in the last section.
Our computations suggest that a majority of the $3330$ non-isomorphic $3$-element groupoids have their ac-spectrum reaching the upper bound $n!C_{n-1}$ and thus have their associative spectrum reaching the upper bound $C_{n-1}$.
Some other $3$-element groupoids have smaller spectra, including those given earlier in this paper as examples for various upper bounds to be sharp.
We also have computational data on the spectra of several other $3$-element groupoids but do not have any general result on them.

For instance, our computations show that the first several terms of the associative spectrum and ac-spectrum of each of the following groupoids are $1,1,2,5,12,28,65,151,351$ and $1,2,12,96,880$, respectively;
the former agrees with the initial terms of a trisection of the Padovan sequence~\cite[A034943]{OEIS}.
\begin{center}
\begin{tabular}{cccccccc}
\begin{tabular}{c|ccc}
$*$ & $0$ & $1$ & $2$ \\
\hline
$0$ & $0$ & $0$ & $0$ \\
$1$ & $1$ & $1$ & $0$ \\
$2$ & $1$ & $0$ & $1$ 
\end{tabular}
&
\begin{tabular}{c|ccc}
$*$ & $0$ & $1$ & $2$ \\
\hline
$0$ & $0$ & $0$ & $1$ \\
$1$ & $1$ & $1$ & $0$ \\
$2$ & $1$ & $0$ & $0$ 
\end{tabular}
&
\begin{tabular}{c|ccc}
$*$ & $0$ & $1$ & $2$ \\
\hline
$0$ & $0$ & $1$ & $1$ \\
$1$ & $0$ & $1$ & $0$ \\
$2$ & $0$ & $0$ & $1$ 
\end{tabular}
& 
\begin{tabular}{c|ccc}
$*$ & $0$ & $1$ & $2$ \\
\hline
$0$ & $0$ & $1$ & $1$ \\
$1$ & $0$ & $1$ & $0$ \\
$2$ & $1$ & $0$ & $0$ 
\end{tabular}
\\
$\SC258$ & $\SC685$ & $\SC1594$ & $\SC1600$  
\end{tabular}
\end{center}
It is clear that $\SC258$ and $\SC685$ are anti-isomorphic to $\SC1594$, $\SC1600$, respectively. 
One can check that $\SC258$ and $\SC685$ both satisfy at least the following identities.
\[  (wx)(yz) \approx (wx)(zy), \quad 
((wx)y)z \approx ((wx)z)y, \quad
(vw)(x(yz)) \approx (vw)((xy)z), \quad
v((wx)(yz)) \approx (v(wx))(yz) \]

Next, consider the following $3$-element groupoids.
\begin{center}
\begin{tabular}{cccccccc}
\begin{tabular}{c|ccc}
$*$ & $0$ & $1$ & $2$ \\
\hline
$0$ & $0$ & $0$ & $2$ \\
$1$ & $2$ & $0$ & $2$ \\
$2$ & $2$ & $2$ & $0$ 
\end{tabular}
&
\begin{tabular}{c|ccc}
$*$ & $0$ & $1$ & $2$ \\
\hline
$0$ & $0$ & $0$ & $2$ \\
$1$ & $2$ & $2$ & $0$ \\
$2$ & $2$ & $0$ & $0$ 
\end{tabular}
&
\begin{tabular}{c|ccc}
$*$ & $0$ & $1$ & $2$ \\
\hline
$0$ & $0$ & $1$ & $1$ \\
$1$ & $1$ & $0$ & $0$ \\
$2$ & $0$ & $0$ & $1$ 
\end{tabular}
&
\begin{tabular}{c|ccc}
$*$ & $0$ & $1$ & $2$ \\
\hline
$0$ & $0$ & $1$ & $1$ \\
$1$ & $1$ & $0$ & $1$ \\
$2$ & $0$ & $1$ & $0$ 
\end{tabular}
\\
$\SC1414$ & $\SC1477$ & $\SC1693$ & $\SC1717$ 
\end{tabular}
\end{center}
There is an anti-isomorphism between $\SC1414$ and $\SC1717$ and between $\SC1477$ and $\SC1693$ by swapping $1$ and $2$.
It is routine to check that $\SC1414$ and $\SC1693$ both satisfy the identities $(wx)(yz) \approx (yz)(wx)$ and $((wx)y)z \approx ((wx)z)y$.
Computations show that the first several terms of its associative spectrum and ac-spectrum are $1,1,2,5,13,35,97,275,794,2327$ and $1,2,12,96,980$;
the former matches with the initial terms of a \emph{generalized Catalan number}~\cite[A025242]{OEIS}, which counts Dyck paths of length $2n$ avoiding $UUDD$.


Computations also show that the first several terms of the associative spectrum and ac-spectrum of the following two anti-isomorphic groupoids are $1,1,2,5,14,42,132,429,1430$ and $1,2,12,108,1340$;
the former agrees with $C_{n-1}$ while the latter is less than $n!C_{n-1}$.
\begin{center}
\begin{tabular}{ccccccc}
\begin{tabular}{c|ccc}
$*$ & $0$ & $1$ & $2$ \\
\hline
$0$ & $0$ & $0$ & $0$ \\
$1$ & $1$ & $0$ & $1$ \\
$2$ & $1$ & $1$ & $1$ 
\end{tabular}
&
\begin{tabular}{c|ccc}
$*$ & $0$ & $1$ & $2$ \\
\hline
$0$ & $0$ & $1$ & $1$ \\
$1$ & $0$ & $0$ & $1$ \\
$2$ & $0$ & $1$ & $1$ 
\end{tabular}
\\
$\SC229$ & $\SC1553$
\end{tabular}
\end{center}
One can check that $\SC229$ satisfies the identity $((wx)y)z \approx ((wx)z)y$.

It would be nice if the associative spectra and ac-spectra of the above $3$-element groupoids (or even better, groupoids satisfying the same identities as the above groupoids) could be determined.

Another question is about the arithmetic mean on $\mathbb{R}$.
Cs\'{a}k\'{a}ny and Waldhauser~\cite{AssociativeSpectra1} showed that its associative spectrum is $C_{n-1}$. 
In previous work~\cite{AC-Spectrum}, we showed that its ac-spectrum is the number of ways to write $1$ as an ordered sum of $n$ powers of $2$~\cite[A007178]{OEIS}. 
It would be interesting to find the identities that could be used to characterize all the groupoids whose associative spectra and ac-spectra are bounded by the above and if possible, find a $3$-element groupoid to achieve the upper bounds.

Lastly, we provide a generalization of a result in our earlier work~\cite{AC-Spectrum}, which asserts that an associative groupoid $(G,*)$ must have $\spac{*}\le n!$ and this upper bound holds as an equality if $(G,*)$ is noncommutative and has an identity element.

\begin{theorem}\label{thm:n!}
For any groupoid $(G,*)$, we have $\spac{*} \leq n! \cdot \spa{*}$. Moreover, this inequality holds as an equality if $(G,*)$ is noncommutative and has an identity element.
\end{theorem}

\begin{proof}
For a bracketing $t \in \B_n$ and a permutation $\sigma \in \SS_n$, let $t$ denote the full linear term obtained by replacing the variable $x_i$ with $x_{\sigma(i)}$ for all $i \in \{1, \dots, n\}$.
Consider two full linear terms in $\F_n$; they can be written as $s_\sigma$ and $t_\tau$, where $s, t \in \B_n$ and $\sigma, \tau \in \SS_n$.
It is clear that if $\sigma = \tau$, then $(s_\sigma)^* = (t_\tau)^*$ if and only if $s^* = t^*$.
The inequality $\spac{*} \leq n! \cdot \spa{*}$ follows immediately from this fact.

Assume now that $(G,*)$ is noncommutative and has a neutral element $0$.
Then there are elements $a, b \in G$ such that $a * b \neq b * a$.
Assume that $\sigma \neq \tau$.
Then there exist $i, j \in \{1, \dots, n\}$ such that $\sigma^{-1}(i) < \sigma^{-1}(j)$ and $\tau^{-1}(i) > \tau^{-1}(j)$.
Let $h \colon X_n \to G$ be the assignment $x_i \mapsto a$, $x_j \mapsto b$ and $x \mapsto 0$ for all $x \in X_n \setminus \{x_i, x_j\}$.
It is easy to see that $h(s_\sigma) = a * b$ and $h(t_\tau) = b * a$; hence $(s_\sigma)^* \neq (t_\tau)^*$.
We conclude that $(s_\sigma)^* = (t_\tau)^*$ if and only if $s^* = t^*$ and $\sigma = \tau$, and the equality $\spac{*} = n! \cdot \spa{*}$ follows.
\end{proof}

\section*{Acknowledgment}

We use Sage (\url{https://www.sagemath.org/}) to help discover and verify the results in this paper.

\end{document}